\documentclass[a4paper]{article}
\usepackage[margin=30mm]{geometry}
\usepackage{amsmath}
\usepackage{amssymb}
\usepackage{amsthm}
\usepackage{braket}
\usepackage{titlesec}
\usepackage{titlefoot}
\usepackage{color}
\usepackage{comment}
\usepackage{bm}

\titleformat*{\section}{\normalsize\bfseries}
\titleformat*{\subsection}{\normalsize\bfseries}
\allowdisplaybreaks
\def\e{{\varepsilon}}        
\def\p{\partial}

\newtheorem{thm}{Theorem}[section]
\newtheorem{lem}[thm]{Lemma}
\newtheorem{dfn}[thm]{Definition}
\newtheorem{cor}[thm]{Corollary}
\newtheorem{prop}[thm]{Proposition}
\newtheorem{rem}[thm]{Remark}

\newcommand{\pt}{\partial}           
\newcommand{\R}{\mathbb R}

\newcommand{\Z}{\mathbb Z}
\newcommand{\N}{\nabla}

\newcommand{\lam}{\lambda}

\newcommand{\dl}{\delta}

\newcommand{\s}{\sigma}

\renewcommand{\t}{\tau}

\newcommand{\hr}{\hookrightarrow}
\newcommand{\wt}{\widetilde }
\newcommand{\wh}{\widehat }
\renewcommand{\div}{{\rm{div\,}}}
\def\<{\langle}
\def\>{\rangle}

\newcommand{\dB}{\dot{B}}

\newcommand{\fB}{\widehat{\dot{B}}}

\newcommand{\ve}{\varepsilon}

\newcommand{\al}{\alpha}
\newcommand{\supp}{\text{supp\;}}

\newcommand{\les}{\lesssim\;}

\newcommand{\F}{\mathcal{F}}

\newcommand{\rd}{\color{red}}
\makeatletter
\@addtoreset{equation}{section}
 
  \makeatother
 
\begin{document}

\title{
\vspace{-1cm}
\large{\bf Asymptotic Profiles for Solutions to the Incompressible \\
Navier--Stokes Equations in the Critical Fourier--Herz Spaces}}
\author{Ikki Fukuda and Ryosuke Nakasato\\ [.7em]
Faculty of Engineering, Shinshu University
}
\date{}
\maketitle


\vspace{-0.75cm}
\begin{abstract}
We consider the large time asymptotic behavior of solutions to the initial value problem for the incompressible Navier--Stokes equations in the whole space $\mathbb{R}^d$ ($d\ge 2$). When the initial velocity belongs to $L^1(\mathbb{R}^d)$, it is well-known that its spatial integral vanishes as a consequence of the divergence-free compatibility condition. Taking this property into account, Carpio (1996) and Fujigaki--Miyakawa (2001) derived higher-order asymptotic formulas for the strong solutions.
In this paper, we investigate the large time behavior of solutions when the initial velocity belongs to the Fourier--Herz space $\widehat{L}^1(\mathbb{R}^d)$, which provides a broader framework than $L^1(\mathbb{R}^d)$ and imposes an additional condition in the low-frequency region. In particular, in the two-dimensional case, we show that an asymptotic profile different from those obtained by Carpio and Fujigaki--Miyakawa arises.
\end{abstract}


\section{Introduction}  

\indent

We consider the initial value problem of the following incompressible Navier--Stokes equations in the $d$-dimensional Euclidean space $\R^{d}$ ($d \ge 2$):  
\begin{equation} \label{eqn;NS} 
\left\{
\begin{split}
&\pt_t u - \mu \Delta u + \div (u \otimes u) + \N p= 0, 
&x \in \R^d, \ t>0, \\
&\div u = 0, &x \in \R^d, \ t>0,\\
&u|_{t=0}=u_0, &x \in \R^d,  
\end{split}
\right.
\end{equation}
where $u=u(x,t): \R^d \times \R_+ \to \R^d$ and $p=p(x, t): \R^d \times \R_+ \to \R$ 
denote the unknown velocity of the fluid and the unknown pressure of the fluid, respectively, while $u_{0}=u_{0}(x):\R^d \to \R^d$ is the given initial velocity. 
Also, $\mu>0$ means the kinematic viscosity and the tensor product $u\otimes u$ is defined by $d$-dimensional square matrix $(u_{i}u_{j})_{i, j}$. 
The purpose of our study is to investigate the large time asymptotic behavior of the solutions to \eqref{eqn;NS}. 
In particular, we aim to obtain the first and second terms of the asymptotic expansion in the Fourier--Herz spaces defined in Definition 1.1. 

First of all, let us introduce some known results related to the global well-posedness for the problem \eqref{eqn;NS}. 
The following argument on the scaling invariances introduced by Fujita--Kato 
\cite{FK64} is the one of the significant approach expected to solve this problem: 
If $(u,p)$ is a solution to the problem \eqref{eqn;NS}, 
the pair of scaled functions $(u_\lam, p_\lam)$ given by 
\begin{equation} \label{eqn;sc}
\left\{
\begin{aligned}
u_\lam(t,x)&=\lam u(\lam^2 t, \lam x), \\
p_\lam(t,x)&=\lam^2 p(\lam^2 t,\lam x),
\end{aligned}
\right. \quad \lam > 0
\end{equation}
also satisfy the same equations. One can find that the invariant function class under the scaling \eqref{eqn;sc} is given by {\it the critical space}, for instance, the critical Bochner--Sobolev spaces are 
\begin{equation*}
  L^r(\R_+;\dot{H}_p^s(\R^d)) \quad {\rm if} \quad \frac{2}{r}+\frac{d}{p}=1+s.  
\end{equation*}

The well-posedness theory for the initial value problem of \eqref{eqn;NS} in critical spaces has been developed progressively from critical Sobolev spaces to critical Lebesgue and Besov spaces. In the critical Sobolev space $\dot{H}^{-1+d/2}(\mathbb{R}^d)$, Fujita--Kato \cite{FK64} established the local well-posedness of \eqref{eqn;NS} and the global well-posedness for sufficiently small initial data. Kato \cite{K84} subsequently extended this result to the critical Lebesgue space $L^d(\mathbb{R}^d)$.
The theory was further developed in the framework of homogeneous Besov spaces $\dot{B}^{-1+d/p}_{p,\infty}(\mathbb{R}^d)$ for $d<p<\infty$; see, for example, \S 5 of \cite{BCD11} for a review. At the endpoint of this scale, Koch--Tataru \cite{KT01} established the global well-posedness for sufficiently small initial data in the critical space $BMO^{-1}(\mathbb{R}^d)$.
On the other hand, the endpoint critical Besov spaces $\dot{B}^{-1}_{\infty,\s}(\mathbb{R}^d)$ exhibit a strikingly different behavior. Bourgain--Pavlovi\'c \cite{BP08}, Yoneda \cite{Y10}, and Wang \cite{W15} established the ill-posedness of \eqref{eqn;NS} in $\dot{B}^{-1}_{\infty,\s}(\mathbb{R}^d)$ for all $1\leq \s\leq\infty$. Thus, while well-posedness extends to the endpoint space $BMO^{-1}(\R^d)$, the corresponding endpoint Besov spaces $\dot{B}^{-1}_{\infty,\s}(\R^d)$ exhibit ill-posedness.

As an alternative approach, Giga--Inui--Mahalov--Saal \cite{GIMS08} and Giga--Saal \cite{GS11} studied the problem \eqref{eqn;NS} with the Coriolis force in the framework of non-decaying initial velocities. They established uniform global solvability in the scaling-invariant space $\F M_{0,\s}^{-1}(\mathbb{R}^d)$.
Several years later, Lei--Lin \cite{LL11} obtained corresponding global well-posedness results in the Lei--Lin space:
$$
\mathcal{X}^{-1}(\mathbb{R}^d)
:=\left\{f\in\mathcal{S}'(\R^d) : \widehat{f} \in L^1_{loc}(\R^d), \,|\cdot|^{-1}\widehat{f}\in L^1_\xi(\mathbb{R}^d)\right\}.
$$
Notice that $\mathcal{X}^{-1}(\mathbb{R}^d)$ is equivalent to $\widehat{\dot{B}}{^{-1}_{\infty,1}}(\mathbb{R}^d)$ in the sense of Lemma \ref{lem;equi} mentioned below.
Subsequently, Cannone--Wu \cite{Ca-Wu} and Iwabuchi--Takada \cite{IT14} further extended the result of Lei--Lin by establishing the global well-posedness of \eqref{eqn;NS} in the critical Fourier--Herz spaces $\widehat{\dot{B}}{^{-1}_{\infty,\s}}(\mathbb{R}^d)$ for all $1\leq \s\leq2$ (see Definition 1.2 below). Moreover, Iwabuchi--Takada \cite{IT14} showed that this range is optimal in the sense that, for all $2<\s\leq\infty$, the problem \eqref{eqn;NS} is ill-posed in $\widehat{\dot{B}}{^{-1}_{\infty,\s}}(\mathbb{R}^d)$.

The large time behavior of solutions to \eqref{eqn;NS} has been extensively studied up to the present. 
In \cite{K84}, Kato established time-decay estimates in $L^q(\mathbb{R}^d)$ $(q>d)$ for the strong solutions to \eqref{eqn;NS} constructed therein, assuming that the initial velocity belongs to $L^d(\mathbb{R}^d)$.
The energy decay of Leray--Hopf weak solutions was first established by Masuda \cite{M84}. By introducing the Fourier splitting method, Schonbek \cite{S80} obtained algebraic decay rates for weak solutions to \eqref{eqn;NS} with large initial data in $L^1(\mathbb{R}^d)\cap L^2(\mathbb{R}^d)$ for $d\geq3$. Subsequently, optimal decay rates were established by Schonbek \cite{S85} and Kajikiya--Miyakawa \cite{KM86}.
In the case where $\langle x\rangle u_0\in L^1(\mathbb{R}^d)$, Wiegner \cite{W87} established an enhanced decay estimate
$
\|u(t)\|_{L^2}=O(t^{-d/4-1/2})
$
for all $d \ge 2$. 
The optimality of Wiegner's decay rate was subsequently established by Miyakawa--Schonbek \cite{MS01}. 
For a more detailed analysis, we refer the reader to \cite{B01, B04, BS18}.

Regarding the asymptotic profiles of solutions to \eqref{eqn;NS}, Carpio \cite{C96} showed that the strong solution to \eqref{eqn;NS} admits the following asymptotic expansion in terms of spatial derivatives of the heat kernel
$G_\mu(t)=G_\mu(x, t):=(4\pi\mu t)^{-d/2}e^{-|x|^2/4\mu t}: $
\begin{equation} \label{eqn;AP}
\begin{aligned}
u(t)&=G_\mu(t)\int_{\mathbb{R}^d}u_0(y)dy
-\partial_kG_\mu(t)\int_{\mathbb{R}^d}y_k u_0(y)dy \\
&\ \ \ -\partial_kS_\mu(t)
\int_0^\infty\int_{\mathbb{R}^d}
(u_k u)(y,\tau)dyd\tau
+o\left(t^{-\frac{d}{2}(1-\frac{1}{p})-\frac{1}{2}}\right) \quad (t \to \infty)
\end{aligned}
\end{equation}
in $L^p(\R^d)$, 
under the assumptions that $d\geq3$, $u_0\in L^d(\mathbb{R}^d)$ satisfies $\operatorname{div}u_0=0$, $\|u_0\|_{L^d}\ll1$, and $\langle x\rangle u_0\in L^1(\mathbb{R}^d)$, together with some additional technical assumptions. 
Here, $S_\mu(t)=S_\mu(x,t)$ denotes the Oseen kernel associated with the Stokes semigroup, defined by 
\begin{align}
S_\mu(x,t)
:=\mathcal{F}^{-1}\left[e^{-\mu t|\xi|^2}P(\xi)\right](x),
\qquad
P(\xi)
:=\left(\delta_{lm}-\frac{\xi_l\xi_m}{|\xi|^2}\right)_{1\leq l,m\leq d}.
\label{DEF-Stokes}
\end{align}
We note that, although the profile functions on the right-hand side of \eqref{eqn;AP} are linear, the effect of the nonlinear term appears in the amplitude of the third term.
On the other hand, when $d=2$, Carpio \cite{C96} also obtained the asymptotic expansion
\begin{equation}\label{eqn;AP2}
u(t)=G_\mu(t)M
-(\log t)\,\partial_kS_\mu(t)\frac{M^kM}{2}
+o\left(t^{-\frac{d}{2}(1-\frac{1}{p})-\frac{1}{2}}\log t\right)
\quad (t\to\infty)
\end{equation}
in $L^p(\mathbb{R}^d)$, assuming that $u_0\in L^2(\mathbb{R}^2)$ satisfies $\operatorname{div}u_0=0$ and $\langle x\rangle u_0\in L^1(\mathbb{R}^2)$, together with some additional technical assumptions. For convenience, in \eqref{eqn;AP2}, we denote the spatial integral of the initial velocity by 
$$
M={}^{\rm t}(M^1,\dots,M^d)
:={}^{\rm t}\left(
\int_{\mathbb{R}^d}u_0^1(y)dy,\dots,
\int_{\mathbb{R}^d}u_0^d(y)dy
\right).
$$
Later on, Fujigaki--Miyakawa \cite{FM01} showed that the asymptotic formula \eqref{eqn;AP} holds for all $d\ge 2$. 
They also observed that $M=0$ whenever $\div u_0=0$ and $u_0\in L^1(\mathbb{R}^d)$ (cf. Miyakawa \cite{M98}), 
so that the first term on the right-hand side of \eqref{eqn;AP} vanishes. 
Here, we also note that the two-dimensional formula \eqref{eqn;AP2} obtained by Carpio \cite{C96} formally contains a logarithmic correction $\log t$ in its second term. However, under the assumptions considered there, one has $M=0$, and hence both the first and second terms on the right-hand side of \eqref{eqn;AP2} vanish identically. Thus, in this setting, \eqref{eqn;AP2} does not provide a nontrivial asymptotic profile in the corresponding order. On the other hand, the result of Fujigaki--Miyakawa \cite{FM01} gives a more refined asymptotic description that remains nontrivial even when $M=0$. For an analogous result in the Fourier--Herz spaces, see also \cite{N26}.

In the aforementioned studies, spatial decay of the initial velocity expressed through assumptions such as $u_0\in L^1(\mathbb{R}^d)$, plays an essential role in the derivation of asymptotic profiles. The purpose of this paper is to investigate the derivation of asymptotic profiles within the framework of Fourier--Herz spaces, without imposing such spatial decay assumptions on the initial velocity.

\medskip
\par\noindent
\textbf{\underline{Basic Notations and Definitions.}} 
\indent

\smallskip
Before stating the main results of this paper, we introduce some 
notations and definitions. 
For $d \ge 1$ and $1\le p\le\infty$, let $L^p = L^p(\R^d)$ be the Lebesgue spaces. 
For any $f$ belonging to the Schwartz class $\mathcal{S}=\mathcal{S}(\R^d)$, 
the Fourier transform of $f$ denoted by $\wh{f}=\wh{f}(\xi)$ or $\mathcal{F}[f]=\mathcal{F}[f](\xi)$ is  
\[
\wh{f}(\xi)= \mathcal{F}[f](\xi):= \frac{1}{(2 \pi)^{\frac{d}{2}}} \int_{\R^d} e^{-i x \cdot \xi}f(x) dx.             
\]
Similarly, for any $g\in \mathcal{S}(\R^d_\xi)$, the Fourier inverse transform $\mathcal{F}^{-1}[g]=\mathcal{F}^{-1}[g](x)$ is defined by 
\[
 \mathcal{F}^{-1}[g](x) := \frac{1}{(2 \pi)^{\frac{d}{2}}} \int_{\R^d} e^{i x \cdot \xi} g(\xi) d\xi.             
\]
Let $\{ \phi_j \}_{j \in \Z}$ be the Littlewood--Paley 
dyadic decomposition of unity, i.e., for a non-negative radially 
symmetric function $\phi \in \mathcal{S}(\mathbb{R}^d)$, we set (for the construction 
of $\{ \phi_j \}_{j \in \Z}$, see e.g. \cite{BCD11, S93}) 
\begin{equation*}
\wh{\phi_j}(\xi) := \wh{\phi}(2^{-j}\xi) \ (j \in \Z), \ \ \sum_{j \in \Z}\wh{\phi}_j(\xi) = 1 \ (\xi \neq 0) \ \ \text{and} \ \ 
\supp \wh{\phi}
\subset
\left\{\xi \in \R^{d} : \frac{1}{2} \le |\xi| \le 2\right\}. 
\end{equation*}

In what follows, we would like to introduce the rigorous definitions of the Fourier--Herz spaces. First, let us define the Fourier--Lebesgue spaces $\wh{L}^p=\wh{L}^p(\R^d)$ as follows: 
\[
  \wh{L}^p(\R^d):=\left\{f \in \mathcal{S}'(\mathbb{R}^d) : \wh{f} \in L^1_{loc}(\R^d),\|f\|_{\wh{L}^p} <\infty\right\} \quad {\rm with} \quad \|f\|_{\wh{L}^p}:=\|\wh{f}\|_{L^{p'}},  
\]
where $\mathcal{S}'=\mathcal{S}'(\R^d)$ is the space of tempered distributions. Here and after, $p'$ denotes the exponent of H\"older conjugate, namely $\frac{1}{p}+\frac{1}{p'}=1$. Next, we shall introduce the Fourier--Herz spaces. 
The following definition is inspired by Gr\"{u}nrock \cite{G04}. 
\begin{dfn}[Homogeneous Fourier--Herz Spaces]
Let $d \ge 1$ and $s \in \R$, $1 \le p,\s \le \infty$. 
Then, we define the homogeneous Fourier--Herz spaces $\fB{_{p,\s}^s} = \fB{_{p,\s}^s}(\R^d)$ as follows:
\begin{align*}
  \fB{_{p,\s}^s}(\R^d)
  :=\left\{f\in\mathcal{S}'(\mathbb{R}^d) : \widehat{f}\in L^1_{loc}(\R^d), \,\|f\|_{\fB{_{p,\s}^s}} < \infty\right\}, \ \ 
  \|f\|_{\fB{_{p,\s}^s}}:=\left\|\left\{2^{sj}\|\dot{\Delta}_j f\|_{\widehat{L}^{p}}\right\}_{j\in\Z}\right\|_{\ell^\s}, 
\end{align*}  
where $\dot{\Delta}_jf:=\mathcal{F}^{-1}[\widehat{\phi}_j \widehat{f}]$ for some 
$f \in \mathcal{S}'(\mathbb{R}^d)$. 
\end{dfn}
\begin{rem}[]
{\rm 
Let $s \in \R$, $d \ge 1$ and $1 \le \s \le \infty$. Then, by virtue of Plancherel's theorem, we can see that  
the spaces $\fB{_{2,\s}^s}$ and $\dB_{2,\s}^s$ coincide in the sence of the norm equivalence, 
where and in what follows, $\dB_{2,\s}^s=\dB_{2,\s}^s(\R^d)$ means the homogeneous Besov spaces defined by 
\[
  \dB^s_{2,\s}(\R^d):=\left\{f \in \mathcal{S}'(\mathbb{R}^d) : \wh{f}\in L^1_{loc}(\R^d), \,\|f\|_{\dB^s_{2,\s}}<\infty\right\}, 
  \ \  
  \|f\|_{\dB^s_{2,\s}}:=\left\|\left\{2^{sj}\|\dot{\Delta}_j f\|_{L^2}\right\}_{j\in\Z}\right\|_{\ell^\s}.  
\]
}
\end{rem}

In addition to the above, the following notations are also used throughout this paper:
\begin{itemize}
\item 
Let $X$ be a Banach space, $I \subset \R$ be an interval and $1\le r \le \infty$. 
One denotes by the Bochner spaces $L^r(I;X)$ the set of strongly measurable functions $f:I\to X$ such that $t \mapsto \|f(t)\|_X$ 
belongs to $L^r(I)$. For $f \in L^r(I;X)$, one defines $\|f\|_{L^r(I;X)}:=\|\|f\|_{X}\|_{L^r(I)}$. 
\item
For $1 \le q<\infty$, $\ell^q(\Z)$ denotes the set of all sequences $\{a_j\}_{j \in \Z}\subset \R$ such that $\sum_{j \in \Z}|a_j|^q<\infty$. 
In the case of $q=\infty$, $\ell^\infty(\Z)$ means the set of all bounded sequences of real numbers. 
\item
For $s \in \R$ and $f \in \mathcal{S}'(\mathbb{R}^d)$, the operators $|\N|^s$ and $\<\N\>^s$ are designating the Riesz potential and the Bessel potential defined by $|\N|^s f:=\mathcal{F}^{-1}[|\xi|^s \wh{f}\,]$ and $\<\N\>^s f:=\mathcal{F}^{-1}[\<\xi\>^s \wh{f}]$ with $\<\xi\>^s:=(1+|\xi|^2)^{\frac{s}{2}}$, respectively. 
In addition, the heat semi-group $e^{\mu t\Delta}$ is defined by $e^{\mu t\Delta}f:=\mathcal{F}^{-1}[e^{-\mu t|\xi|^{2}}\wh{f}]$. Similarly, $e^{\sqrt{\mu t}|\N|}f$ can be defined by $e^{\sqrt{\mu t}|\N|}f:=\mathcal{F}^{-1}[e^{\sqrt{\mu t}|\xi|} \widehat{f}]$. 
\item
Throughout this paper, for the functions $f(t)$ and $g(t)$, $f(t)\les g(t)$ means that there exists a positive constant $C>0$ independent of $t>0$ such that the inequality $f(t)\le Cg(t)$ holds. 
\item
Throughout this paper, we employ the Einstein summation convention with respect to repeated indices. 
Namely, whenever an index appears twice in a term, summation over that index from $1$ to $d$ is implicitly understood. More precisely, we write $A_kB_k
:= \sum_{k=1}^d A_kB_{k}$. 
\end{itemize}

\medskip
\par\noindent
\textbf{\underline{Fundamental Results.}} 
\indent

\smallskip
At the end of this section, let us explain some fundamental results that will be used in stating our main results. We first introduce the definition of the solutions used in this study.

\begin{dfn}[Global Mild Solution] 
If the function $u \in L^2(\R_+;\wh{L}^\infty)$ satisfies the following integral equations, 
then we call $u$ the {\it global mild solution} of the problem \eqref{eqn;NS}:  
\begin{equation}\label{IE}
u(t)=e^{\mu t \Delta}u_0-\int_0^t e^{\mu(t-\t)\Delta} \mathcal{P}_\s \div (u \otimes u)(\t) d\t, 
\end{equation}
where $\mathcal{P}_\s$ denotes the Helmholtz decomposition 
which is defined by $\mathcal{P}_\s:=Id+(-\Delta)^{-1}\N \div$. 
\end{dfn}

In \cite{N26}, analyticity and time-decay estimates for solutions to \eqref{eqn;NS} were established in the endpoint critical Fourier--Herz space $\wh{\dot{B}}{_{\infty,2}^{-1}}(\R^d)$. 
We recall the results below. 

\begin{prop}[Global Well-posedness and Analyticity, {\rm \cite{N26}}] \label{thm;GWP} 
Let $d \ge 2$. Suppose that the initial velocity $u_0$ satisfies 
$$ 
u_0 \in \widehat{\dot{B}}{_{\infty,2}^{-1}}(\R^d) \quad {\it and } \quad 
\div u_0 =0 \ \ in \ \,\mathcal{S}'.
$$
If in addition, $\|u_0\|_{\widehat{\dot{B}}{_{\infty,2}^{-1}}}$ is sufficiently small, then the problem \eqref{eqn;NS} admits a unique global mild solution  $u \in L^2(\R_+;\widehat{L}^\infty)$. 
In particular, the solution $u$ fulfills 
$e^{\sqrt{\mu t}|\N|}u \in L^2(\R_+;\widehat{L}^\infty)$. 
\end{prop}

\begin{prop}[$L^p$-$L^1$ Type Time-decay Estimate, \cite{N26}] \label{thm;decay} 
Let $d\ge2$ and $1 \le p \le \infty$. 
Suppose that the initial velocity $u_0$ satisfy the same assumptions as in Proposition~\ref{thm;GWP} and 
$u$ denotes the corresponding global mild solution of the problem \eqref{eqn;NS}. 
If we additionally assume $u_0 \in \wh{\dot{B}}{_{1,\infty}^0}(\R^d)$,  
then for all $s>-d/p'$, the global mild solution $u$ fulfills the following decay estimate:   
\begin{equation} \label{est;LpL1}
\||\N|^s u(t)\|_{\fB{_{p,1}^0}} \les t^{-\frac{d}{2}(1-\frac{1}{p})-\frac{s}{2}}, \ \ t>0. 
\end{equation}
In particular, if $2\le p \le \infty$, it holds true that $\||\N|^s u(t)\|_{\dot{B}{_{p,1}^0}} \les t^{-\frac{d}{2}(1-\frac{1}{p})-\frac{s}{2}}$ for $t>0$. 
\end{prop}

\begin{rem}  
{\rm
From the definitions of the Fourier--Herz spaces and the Fourier--Lebesgue spaces, we can easily check that the fact $\wh{\dot{B}}{_{p,1}^0}(\R^{d}) \hookrightarrow \wh{L}^{p}(\R^d)$ holds true for all $1\le p\le \infty$. Therefore, it directly follows from \eqref{est;LpL1} that the following estimate holds: 
\begin{equation}\label{est;LpL1-re}
\||\N|^s u(t)\|_{\wh{L}^{p}} \les t^{-\frac{d}{2}(1-\frac{1}{p})-\frac{s}{2}}, \ \ t>0. 
\end{equation}
}
\end{rem}

The above Proposition~\ref{thm;decay} and \eqref{est;LpL1-re} claim that the $L^p$-$L^1$ decay estimate as in \cite{S85, S86} holds true for the 
$\widehat{\dot{B}}{_{\infty,2}^{-1}}(\R^d)$-solution. 
Here, let us focus on the auxiliary assumption $u_0 \in \widehat{\dot{B}}{_{1,\infty}^0}(\R^d)$. 
By virtue of the Riemann--Lebesgue theorem and Lemma \ref{lem;equi} below, 
we can immediately see that 
$$L^1(\R^d) \hr \widehat{L}^1(\R^d) \simeq \widehat{\dot{B}}{_{1,\infty}^0}(\R^d)$$ 
holds true. 
Generally, the following relations hold between the Lebesgue spaces $L^p(\R^d)$ and 
Fourier--Lebesgue spaces $\widehat{L}^p(\R^d)$: 
\[
L^p(\R^d) \hr \widehat{L}^p(\R^d) \quad \text{if} \quad p \le 2, \qquad 
\widehat{L}^p(\R^d) \hr L^p(\R^d) \quad \text{if} \quad p \ge 2. 
\]
Then, thanks to this embeddings, 
we could obtain more sharp $L^p$-$L^1$ decay estimate in the case of $p \ge 2$ because that for all $t>0$, it holds that 
\[
t^{\frac{d}{2}(1-\frac{1}{p})} \|u(t)\|_{L^p} \les t^{\frac{d}{2}(1-\frac{1}{p})} \|u(t)\|_{\widehat{L}^p} \les \|u_0\|_{\widehat{L}^1} \les \|u_0\|_{L^1}. 
\]

\section{Main Results}  

\indent

In this section, we would like to state our main results in this paper. 
First, we shall introduce a key assumption which is essential for stating the asymptotic formulas of the solution $u$ to \eqref{eqn;NS}, in the case where the initial velocity $u_{0}$ belongs to $\wh{L}^{1}(\R^{d})$. Here, noticing that $u_{0}\in \wh{L}^{1}(\R^{d})$ means that its Fourier transform $\wh{u}_{0}$ is essentially bounded on $\mathbb{R}^d$. In our approach, we need to prescribe the information on $\wh{u}_{0}$ near the origin $|\xi| \to 0$. More precisely, we consider the following assumption: 

\medskip
\noindent
\underline{\bf Assumption (A).}
For the initial velocity $u_{0}$ satisfying $u_{0}\in \wh{L}^1(\R^{d})$, suppose that there exists a non-zero vector $M_{0}={}^{\rm t}(M_{0}^{1}, \cdots, M_{0}^{d})\in \R^{d}$ satisfying the following condition: 
\begin{equation}\label{cond-1}
\lim_{|\xi|\to 0}\left|\wh{u}_{0}(\xi)-P(\xi)M_{0}\right|=0.
\end{equation}

In fact, this type of the condition naturally appears when one describes asymptotic profiles via the low-frequency behavior of the Fourier transform, which corresponds to the far-field behavior in the physical space. We note that this condition is inspired by an idea used by Narazaki--Nishihara \cite{NN08} (see, also \cite{F19}). 
In addition, the appearance of the term $P(\xi)M_{0}$ is naturally related to the divergence-free condition $\div u_0 =0$. Indeed, since the initial velocity field $u_{0}$ is required to be divergence-free, its Fourier transform $\wh{u}_{0}$ must be orthogonal to $\xi$, i.e., $\xi \cdot \wh{u}_{0}(\xi)=0$ holds for all $\xi \neq0$. Hence, the limiting profile of $\wh{u}_{0}(\xi)$ as $|\xi|\to 0$ should lie in the range of the Helmholtz projection $P(\xi)$. In this sense, the vector $M_{0}$ can be interpreted as a generalized moment associated with the low-frequency behavior of the initial data.

\medskip
To investigate the more detailed asymptotic behavior, we need an additional assumption that is slightly stronger than {\rm (A)}. 
The following condition provides a rate of convergence in \eqref{cond-1}.

\medskip
\noindent
\underline{\bf Assumption (B).}
Suppose that the all conditions appeared in the assumption {\rm (A)} hold true. In addition, we also assume that there exists $\alpha>0$ such that the initial velocity $u_{0}$ satisfy 
\begin{equation}\label{cond-2}
\left|\wh{u}_{0}(\xi)-P(\xi)M_{0}\right| =O\left(|\xi|^{\al}\right) \quad \text{as} \quad |\xi|\to 0. 
\end{equation}

In what follows, when we investigating the asymptotic behavior of the solution to \eqref{eqn;NS} under this assumption, let us work under the following conditions on the parameters. 

\medskip
\noindent
\underline{\bf Condition (${\mathbf C}$).} \;Let $d \ge 2$, $s \in \R$ and $1 \le p \le \infty$ satisfy
the followings: 
\begin{align*}
\text{If} \ \ p=1, \ \ s\ge0. \qquad 
\text{If} \ \ 1<p<\infty, \ \ s>-\min\left\{\frac{d}{p'}, 1+\frac{d}{p}\right\}. \qquad 
&\text{If} \ \ p=\infty, \ \ s\ge-1. 
\end{align*}

Now, let us finally state our main results. 
First, we shall give the results that the first asymptotic profile of the solution $u$ to \eqref{eqn;NS} can be given by $S_{\mu}(t)M_{0}$, under the weak assumption {\rm (A)}.  
\begin{thm}[{1st Order Asymptotics, under Weak Assumptions}] \label{thm;1st-1} 
Let $d\ge2$, $s \in \R$ and $1\le p\le \infty$. Suppose that the initial velocity $u_{0}$ satisfy $u_0 \in \widehat{\dot{B}}{_{\infty,2}^{-1}}(\R^d)\cap \wh{\dot{B}}{_{1,\infty}^0}(\R^d)$ with 
$\div u_0 =0$ in $\mathcal{S}'$, and $\|u_0\|_{\widehat{\dot{B}}{_{\infty,2}^{-1}}}$ is sufficiently small. 
Moreover, we assume that $u_{0}$ also fulfill the assumption {\rm (A)}, and the parameters 
$d$, $p$ and $s$   satisfy the condition ${\rm (C)}$.
Then, for the global mild solution $u$ to \eqref{eqn;NS}, the following asymptotic formula holds:
\begin{equation}\label{1st-1}
\lim_{t\to \infty}t^{\frac{d}{2}(1-\frac{1}{p})+\frac{s}{2}}\left\||\N|^{s}\left(u(t)-S_{\mu}(t)M_{0}\right)\right\|_{\wh{L}^{p}}=0. 
\end{equation}
\end{thm}

Moreover, if we additionally assume the strong assumption {\rm (B)}, then the asymptotic rate given in \eqref{1st-1} can be improved. More precisely, we have the following asymptotic formula: 
\begin{thm}[1st Order Asymptotics, under Strong Assumptions] \label{thm;1st-2} 
Suppose that the all conditions as in Theorem~\ref{thm;1st-1} are satisfied. In addition, we assume that the initial velocity $u_{0}$ also fulfill the assumption {\rm (B)}, and the parameters 
$d$, $p$ and $s$ satisfy the condition ${\rm (C)}$. 
Then, for the global mild solution $u$ to \eqref{eqn;NS} with $t\ge 1$, the following asymptotic formula holds:
\begin{equation}\label{1st-2}
\left\||\N|^{s}\left(u(t)-S_{\mu}(t)M_{0}\right)\right\|_{\wh{L}^{p}}
\les 
\begin{cases}
t^{-\frac{d}{2}(1-\frac{1}{p})-\frac{s}{2}-\frac{\min\{\al, 1\}}{2}}, &\text{if} \quad d\ge3, \\
t^{-\frac{d}{2}(1-\frac{1}{p})-\frac{s}{2}}\left\{t^{-\frac{\al}{2}}+t^{-\frac{1}{2}}\log (1+t)\right\}, &\text{if} \quad d=2. 
\end{cases}
\end{equation}
\end{thm}

\begin{rem}  
{\rm
It is worth emphasizing that the asymptotic formulas \eqref{1st-1} and \eqref{1st-2} exhibit a leading profile essentially different from that in the Fujigaki--Miyakawa type expansion  \eqref{eqn;AP}. Indeed, in the present setting, the leading term is given by the Oseen profile $S_\mu(t) M_0$ with the vector $M_{0}$ which characterizes the low-frequency behavior of the initial velocity $u_0$. This is in sharp contrast to the classical $L^1$-setting, where the mass $M=\int_{\mathbb{R}^{d}}u_{0}(x)dx$ necessarily vanishes for divergence-free initial data, so that the corresponding leading term $G_{\mu}(t)M$ appearing in \eqref{eqn;AP} disappears. However, in the present Fourier--Herz framework, the generalized low-frequency quantity $M_0$ need not vanish, and therefore $S_\mu(t) M_0$ survives as a nontrivial leading-order asymptotic profile.
}
\end{rem}

Next, we would like to present an asymptotic formula related to the second-order asymptotics of the solution. The second asymptotic profile obtained below in higher-dimensions $d\ge3$ involves the same function that also appears in Fujigaki--Miyakawa type asymptotic formula \eqref{eqn;AP}.
\begin{thm}[2nd Order Asymptotics, Higher-dimensional Case] \label{thm;2nd-high} 
Let $d\ge3$ and $1<p \le \infty$. Suppose that the all conditions as in Theorem~\ref{thm;1st-2} are satisfied. 
If we additionally assume $\alpha>1$, then for the global mild solution $u$ to \eqref{eqn;NS}, the following asymptotic formula holds:
\begin{equation}\label{2nd-1}
\lim_{t\to \infty}t^{\frac{d}{2}(1-\frac{1}{p})+\frac{s+1}{2}} 
\left\||\N|^{s}\left(u(t)-S_{\mu}(t)M_{0}+\p_{k}S_{\mu}(t)\int_{0}^{\infty}\int_{\R^{d}}(u_{k}u)(x, t)dxdt\right)\right\|_{\wh{L}^{p}}=0. 
\end{equation}
\end{thm}

\begin{rem}  
{\rm

We also note that, in  \eqref{eqn;AP}, the terms involving $\partial_k G_\mu(t)$ and $\partial_k S_\mu(t)$ have the same decay order and thus together form the leading-order asymptotic profile when $M=0$. On the other hand, in our expansion \eqref{2nd-1}, the first and second asymptotic profiles have distinct decay orders, with the latter decaying faster by a factor of $t^{-1/2}$. Thus, \eqref{2nd-1} exhibits a hierarchical asymptotic structure that is different from the Fujigaki--Miyakawa type expansion \eqref{eqn;AP}.
}
\end{rem}

Finally, let us introduce the second-order asymptotic formula of the solution in two-dimension $d=2$. 
In contrast to Carpio's formula \eqref{eqn;AP2} mentioned above, the result below explicitly identifies a nontrivial logarithmic correction as the second-order asymptotic profile. 
\begin{thm}[2nd Order Asymptotics, Two-dimensional Case] \label{thm;2nd-2} 
Let $d=2$ and $1 \le p \le \infty$. Suppose that the all conditions as in Theorem~\ref{thm;1st-2} are satisfied. 
If we additionally assume $\al \ge1$, then for the global mild solution $u$ to \eqref{eqn;NS}, the following asymptotic formula holds:
\begin{equation}\label{2nd-2}
\left\||\N|^{s}\left(u(t)-S_{\mu}(t)M_{0}+(\log t)\,\p_{k}S_{\mu}(t)\mathcal{A}_{k}[M_{0}]\right)\right\|_{\wh{L}^{p}}
\les t^{-\frac{d}{2}(1-\frac{1}{p})-\frac{s+1}{2}}, \ \ t\ge2, 
\end{equation}
where the vector $\mathcal{A}_{k}[M_{0}]$ is defined by 
\begin{equation}\label{DEF-A}
\mathcal{A}_{k}[M_{0}]:=\frac{\pi}{16\mu}\left(2M_{0}^{k}M_{0}+|M_{0}|^{2}e_{k}\right). 
\end{equation}
\end{thm}

\begin{rem}  
{\rm
In view of the second asymptotic expansion, we are able to obtain the optimal asymptotic rate to the first asymptotic profile $S_{\mu}(t)M_{0}$ by using Theorems \ref{thm;2nd-high} and \ref{thm;2nd-2}. Especially, under the assumptions in Theorem \ref{thm;2nd-2}, the solution $u(x,t)$ satisfies the following result:
  \begin{equation*}
   \left\||\N|^{s}\left(u(t)-S_{\mu}(t)M_{0}\right)\right\|_{\wh{L}^{p}}
   =\left(C_*+o(1)\right)t^{-\frac{d}{2}\left(1-\frac{1}{p}\right)-\frac{s+1}{2}}\log t \quad (t\to \infty)
  \end{equation*}
  for any $1\le p \le \infty$, where $C_*:=\left\||\N|^{s}\p_{k}S_{\mu}(1)\mathcal{A}_{k}[M_{0}]\right\|_{\wh{L}^{p}}\neq0$. 
  Therefore, the optimal convergence rate to $S_{\mu}(t)M_{0}$ is $O(t^{-(d/2)(1-1/p)-(s+1)/2})\log t$, which in particular shows that the logarithmic factor appearing in \eqref{1st-2} with $d=2$ is essential. The logarithmic correction in the second asymptotic profile reflects the critical nature of the quadratic nonlinearity $u\otimes u$ with respect to time integrability in two space dimension. It is worth noting that analogous logarithmic corrections have been observed in related parabolic equations with similar nonlinear structures, such as convection-diffusion equation, generalized Burgers equation, and parabolic systems of chemotaxis; see, for example, \cite{EZ91, KM07, KM09, NY07, Z93} and the references therein. Our result shows that such a logarithmic correction also arises for the two-dimensional Navier--Stokes equations in the Fourier--Herz framework.
 }
\end{rem}

The rest of this paper is organized as follows. 
First, in Section~3, we prepare some basic results related to the elemental properties and inequalities on the Fourier--Herz spaces. 
Next, we give some asymptotic formulas for the linear solution $e^{\mu t \Delta}u_{0}$ in Section~4. The proofs of the main results are given in Section~5. 
This section is divided into two subsections below. Subsection~5.1 is devoted to deriving the first asymptotic profile of the solution $u$ to \eqref{eqn;NS}. 
We prove Theorems~\ref{thm;1st-1} and \ref{thm;1st-2} in this subsection. Finally, the second terms of asymptotics are obtained in Subsection~5.2. 
Namely, the proofs of Theorems~\ref{thm;2nd-high} and \ref{thm;2nd-2} at the end of this paper. 
One of the main novelty of this paper lies in Theorem~\ref{thm;2nd-2}, whose most significant contribution is the derivation of the second asymptotic profile of the solution $u$ in two-dimension $d=2$. The asymptotic profile obtained here is essentially different from that of the Fujigaki--Miyagawa type asymptotic expansion.

\section{Preliminaries}  

\indent

In this section, we would like to introduce some preliminary results that play an important role in the proofs of our main results. 
First, let us describe the definition of the Fourier--Sobolev spaces $\wh{\dot{H}}{_p^s}=\wh{\dot{H}}{_p^s}(\R^d)$ and the norm equivalence between these spaces and the Fourier--Herz spaces defined above. For the proof of the following lemma, see e.g. \cite{C18,KY11}. 
\begin{lem}[\cite{C18,KY11}] \label{lem;equi}
Let $d \ge 1$, $1 \le p \le \infty$ and $s \in \R$. Then, the fact $\fB{_{p,p'}^s}(\R^d) \simeq \wh{\dot{H}}{_p^s}(\R^d)$ holds, 
in the sense of the norm equivalence. 
Here, the Fourier--Sobolev spaces $\wh{\dot{H}}{_p^s}=\wh{\dot{H}}{_p^s}(\R^d)$ is defined by 
\[
  \wh{\dot{H}}{_p^s}(\R^d)
  :=\left\{f \in \mathcal{S}'(\mathbb{R}^d) : \wh{f} \in L^1_{loc}(\R^d),\|f\|_{\wh{\dot{H}}{_p^s}}<\infty\right\} \quad \text{with} \quad 
  \|f\|_{\wh{\dot{H}}{_p^s}}:=\||\N|^sf\|_{\wh{L}^p}. 
\]
\end{lem}

The second topic is the product estimate in the Fourier--Herz spaces. 
Here, let us state the standard bilinear estimate without their proof (for the general statement and their proof, see e.g. Lemma~2.4 in \cite{MNO21}).  
\begin{lem}[Bilinear estimate, \cite{MNO21}] \label{lem;bil}
Let $s > 0$ and $1 \le p,\s \le \infty$. 
Then, there exists a positive constant $C > 0$ such that the following estimate holds: 
\begin{equation*} 
\|fg\|_{\fB{_{p,\s}^s}}
\le 
C\left(\|f\|_{\wh{L}^\infty}\|g\|_{\fB{^s_{p,\s}}} 
+ \|f\|_{\fB{^s_{p,\s}}}\|g\|_{\wh{L}^\infty}\right). 
\end{equation*}
\end{lem} 

Next, we would like to mention the product estimate in the Fourier--Herz spaces. 
The result corresponding to the following estimate in the Besov spaces is well-known as in \cite{BCD11}. 
Some similar estimates in the Fourier--Herz spaces are established in \cite{N-pre} (see, also \cite{N22}). 
\begin{lem}[Product estimate, \cite{N-pre, N22}] \label{lem;A-P}
Let $d \ge 1$ and $1 \le p,\s \le \infty$.  
If $s \in \R$ satisfies $|s|<d/p$ for $2 \le p$ and $-d/p'<s<d/p$ for $1\le p<2$, 
then there exists a positive constant $C>0$ such that the following estimate holds: 
\begin{equation*} 
\|fg\|_{\fB{_{p,\s}^{s}}}
\le C\|f\|_{\fB{_{p,\s}^s}} 
     \|g\|_{\wh{L}^{\infty}\cap\fB_{p,\infty}^\frac{d}{p}}.  
\end{equation*}
\end{lem}

Finally, we need to prepare a weighted estimate for the Oseen kernel $S_\mu(x,t)$ in two-dimensional case of $d=2$. The following estimate \eqref{S-west} will be used in the proof of Proposition~\ref{prop-N-asymp-2} below. 
\begin{lem}\label{lem;S}
Let $d=2$. Then, for the Oseen kernel $S_\mu(x,t)$ defined by \eqref{DEF-Stokes}, we have 
\begin{equation}\label{S-west}
\int_{\R^{2}}|x|\left|S_{\mu}(x, t)\right|^{2}dx \les t^{-\frac{1}{2}}, \ \ t>0. 
\end{equation}
\end{lem}
\begin{proof}
First, it follows from the scaling argument and the change of variables that 
\begin{align*}
\int_{\R^{2}}|x|\left|S_{\mu}(x, t)\right|^{2}dx 
= \int_{\R^{2}}|x|\left|t^{-1}S_{\mu}\left(\frac{x}{\sqrt{t}}, 1\right)\right|^{2}dx
=t^{-\frac{1}{2}} \int_{\R^{2}}|y|\left|S_{\mu}(y, 1)\right|^{2}dy, \ \ t>0. 
\end{align*}
Then, in order to obtain \eqref{S-west}, we need to prove 
\begin{equation}\label{Sy-est}
\int_{\mathbb{R}^2}|y|\left|S_\mu(y,1)\right|^2 dy < \infty.
\end{equation}

In what follows, let us prove \eqref{Sy-est}. Now, recalling the definitions of the kernels, i.e., \eqref{DEF-Stokes} and the surrounding discussion, we can see that $G_\mu(y, 1)=\mathcal{F}^{-1}
[e^{-\mu|\xi|^2}](y)
=(4\pi\mu)^{-1}e^{-\frac{|y|^2}{4\mu}}$ and 
\[
P(\xi)
=
I-\frac{\xi\otimes\xi}{|\xi|^2} 
\]
hold, then $S_{\mu}(y, 1)$ satisfies the following differential equation: 
\[
S_\mu(y,1)
=
G_\mu(y, 1)I+\nabla^2\Phi_\mu(y) \quad \text{with} \quad \Phi_\mu
:=
(-\Delta)^{-1}G_\mu(\cdot, 1),
\]
where $\nabla^2 \Phi_\mu(y)$ means the Hessian matrix of $\Phi_\mu$ at $y\in \R^{2}$. Since $G_\mu(\cdot, 1)$ is radial, $\Phi_\mu$ is also radial.
Hence, writing $r=|y|$ and $\Phi_\mu(y)=\phi_\mu(r)$,
we have
\[
-\left(
\phi_\mu''(r)
+
\frac1r\phi_\mu'(r)
\right)
=
G_\mu(r, 1).
\]
Then, multiplying $r$ on the both sides, we obtain $-\left(r\phi_\mu'(r)\right)'=rG_\mu(r, 1)$. Therefore, integrating it over $(0,r)$ and calculating the obtained result, we get  
\[
\phi_\mu'(r)
=
-\frac1r
\int_0^r
\rho G_\mu(\rho, 1)d\rho=\frac{1}{2\pi r}\left(e^{-\frac{r^2}{4\mu}}-1\right).
\]
Thus, the following estimate has been established: 
\[
|\phi_\mu''(r)|\les G_\mu(r, 1)+\left|\frac1r\phi_\mu'(r)\right|\les
e^{-\frac{r^2}{4\mu}}+
\frac{1-e^{-\frac{r^2}{4\mu}}}{r^2}.
\]
Finally, since the Hessian of a radial function satisfies
\[
\left|\nabla^2\Phi_\mu(y)\right| \les \left|\phi_\mu''\left(|y|\right)\right|
+\left|\frac{\phi_\mu'(|y|)}{|y|}\right|,
\]
for some constant $c>0$, we finally obtain the following estimate: 
\[
|S_\mu(y,1)|\les e^{-c|y|^2}+\frac{1-e^{-c|y|^2}}{|y|^2}. 
\]

By virtue of the above pointwise estimate, we can eventually conclude that 
\[
|S_\mu(y,1)|\les1 \ \ \text{for} \ \ |y|\le 1 \quad \text{and} \quad |S_\mu(y,1)| \les |y|^{-2} \ \ \text{for} \ \ |y|\ge 1.
\]
Therefore, using these estimates and the polar coordinates, we arrive at 
\begin{align*}
\int_{\mathbb{R}^2}
|y|\left|S_\mu(y,1)\right|^2dy=\left(\int_{|y|\le 1}+\int_{|y|\ge 1}\right)|y|\left|S_\mu(y,1)\right|^2dy 
\les \int_{|y|\le 1}|y|\,dy+\int_1^\infty r^{-2}dr<\infty.
\end{align*}
This completes the proof of \eqref{Sy-est}. Thus, the desired result \eqref{S-west} has been obtained. 
\end{proof}

\section{Linear Analysis}  

\indent

In this section, we would like to analyze the asymptotic behavior of the linear solution $e^{t\mu \Delta}u_{0}$ to the original problem \eqref{eqn;NS} in our Fourier--Herz setting. 
First, under the conditions {\rm (A)} and {\rm (B)}, we shall derive an approximate form of the solution, which serves as a precursor to its asymptotic profile.
\begin{prop}[Linear Asymptotics, under Weak Assumptions]\label{prop-lin-half1-1}
Let $d\ge2$, $1\le p\le \infty$ and $s\in \R$. Suppose that the initial velocity $u_{0}$ satisfy the assumption {\rm (A)}. 
Moreover, we assume that $d$, $p$ and $s$ satisfy the following conditions: 
\begin{align}
\text{If} \ \ p>1, \ \ s>-\frac{d}{p'} \quad \text{and} \quad 
&\text{if} \ \ p=1, \ \ s\ge 0.  \label{cond-parameter-1-2}
\end{align}
Then, we have the following asymptotic formula: 
\begin{equation}\label{lin-half1}
\lim_{t \to \infty}t^{\frac{d}{2}(1-\frac{1}{p})+\frac{s}{2}}\left\||\N|^{s}\left(e^{\mu t\Delta}u_{0}-S_{\mu}(t)M_{0}\right)\right\|_{\wh{L}^{p}}=0.  
\end{equation}
\end{prop}

\begin{prop}[Linear Asymptotics, under Strong Assumptions]\label{prop-lin-half1-2}
Let $d\ge2$, $1\le p\le \infty$, $s\in \R$ and $\al>0$. Suppose that the initial velocity $u_{0}$ satisfy the assumption {\rm (B)} with $\al>0$. 
Moreover, we assume that $d$, $p$ and $s$ satisfy \eqref{cond-parameter-1-2}. 
Then, we have the following asymptotic formula: 
\begin{equation}\label{lin-half2}
\left\||\N|^{s}\left(e^{\mu t\Delta}u_{0}-S_{\mu}(t)M_{0}\right)\right\|_{\wh{L}^{p}}
\les t^{-\frac{d}{2}(1-\frac{1}{p})-\frac{s+\al}{2}}, \ \ t>0. 
\end{equation}
\end{prop}

\begin{proof}[\bf Proof of Propositions~\ref{prop-lin-half1-1} and \ref{prop-lin-half1-2}]
First, we shall give the proof of \eqref{lin-half2}. From the condition \eqref{cond-2}, we note that there exists $\delta>0$ such that if $|\xi|\le \delta$, then the following result holds: 
\begin{equation}\label{cond-2-re}
\left|\wh{u}_{0}(\xi)-P(\xi)M_{0}\right| \les |\xi|^{\al}. 
\end{equation}

Now, we would like to deal with the case of $p>1$. For the latter sake, by using $\delta>0$ appeared in the above, let us split the target integral as follows: 
\begin{equation} 
\begin{aligned}
&\left\||\xi|^{s}e^{-\mu t|\xi|^{2}}\left(\wh{u}_{0}(\xi)-P(\xi)M_{0}\right)\right\|_{L^{p'}_{\xi}}^{p'}  \\
&=\int_{\R^{d}}|\xi|^{sp'}e^{-\mu p't|\xi|^{2}}\left|\wh{u}_{0}(\xi)-
P(\xi)M_{0}\right|^{p'}d\xi \\
&=\left(\int_{|\xi|\le \delta}+\int_{|\xi|>\delta}\right)|\xi|^{sp'}e^{-\mu p't|\xi|^{2}}\left|\wh{u}_{0}(\xi)-P(\xi)M_{0}\right|^{p'}d\xi \\
&=:I_{1}(t)+I_{2}(t).  \label{DEF-I1-I2}
\end{aligned}
\end{equation}
In what follows, let us evaluate $I_{1}(t)$ and $I_{2}(t)$. First, we shall treat $I_{1}(t)$. 
Since the assumption $s>-d/p'$ in \eqref{cond-parameter-1-2} implies that $(s+\al)p'>-d$ due to 
$\al>0$, we can easily obtain 
\begin{equation}\label{est-I1:p>1}
\begin{aligned}
I_{1}(t) 
&\les \int_{|\xi|\le \delta}|\xi|^{sp'+\al p'}e^{-\mu p't|\xi|^{2}} d\xi \\
&\les \int_{\R^{d}}|\xi|^{(s+\alpha)p'}e^{-\mu p't|\xi|^{2}} d\xi 
\les t^{-\frac{d}{2}-\frac{p'}{2}(s+\alpha)}, \ \ t>0.  
\end{aligned}
\end{equation}

Next, let us deal with $I_{2}(t)$. Here, noticing that $|P(\xi)| \les 1$ holds for all $\xi \in \R^{d}$. Then, from the assumption {\rm (A)}, we have 
\[
\left|\wh{u}_{0}(\xi)-P(\xi)M_{0}\right| \les \left|\wh{u}_{0}(\xi)\right|
+|P(\xi)M_{0}| \les 1, \ \ \text{for} \ \ {\rm a.e.} \ \xi\in \R^{d}. 
\]
Therefore, under the parameter condition \eqref{cond-parameter-1-2}, a direct calculation yields that 
\begin{equation}\label{est-I2:p>1}
\begin{aligned} 
I_{2}(t) &\les  \int_{|\xi|>\delta}|\xi|^{sp'}e^{-\mu p't|\xi|^{2}} d\xi \\
&\les e^{-\frac{\mu}{2}p'\dl^2 t} \int_{|\xi|>\dl} |\xi|^{sp'}e^{-\frac{\mu}{2} p't|\xi|^{2}} d\xi 
\les e^{-\frac{\mu}{2}p'\dl^2 t}. 
\end{aligned}
\end{equation}
Finally, combining \eqref{DEF-I1-I2}, \eqref{est-I1:p>1} and \eqref{est-I2:p>1}, we arrive at 
\begin{equation}\label{est-p>1}
\left\||\N|^{s}\left(e^{\mu t\Delta}u_{0}-S_{\mu}(t)M_{0}\right)\right\|_{\wh{L}^{p}}
\les  t^{-\frac{d}{2p'}-\frac{s+\alpha}{2}}+e^{-\frac{\mu}{2} \delta^{2}t}
\les t^{-\frac{d}{2p'}-\frac{s+\alpha}{2}}, \ \ t>0. 
\end{equation}
This estimate corresponds to the desired decay estimate \eqref{lin-half2} in the case of $p>1$. 

On the other hand, for \eqref{lin-half2} with $p=1$, one can calculate the maximum value of $|\xi|^{s+\al}e^{-\mu t|\xi|^{2}}$ for $\xi \in \R^{d}$ directly by virtue of $s+\al \ge0$, 
and then obtain the following result: 
\begin{equation} \label{est-p=1}
\begin{aligned}
&\left\||\N|^{s}\left(e^{\mu t\Delta}u_{0}-S_{\mu}(t)M_{0}\right)\right\|_{\wh{L}^{1}} 
=\left\||\xi|^{s}e^{-\mu t|\xi|^{2}}\left(\wh{u}_{0}(\xi)-P(\xi)M_{0}\right)\right\|_{L^{\infty}_{\xi}} \\
&\les \left\||\xi|^{s+\al}e^{-\mu t|\xi|^{2}}\right\|_{L^{\infty}_{\xi}(|\xi|\le \delta)}
+\left\||\xi|^{s}e^{-\mu t|\xi|^{2}}\right\|_{L^{\infty}_{\xi}(|\xi|> \delta)} \\
&\les  t^{-\frac{s+\al}{2}}+e^{-\frac{\mu}{2} \delta^{2}t}
\les t^{-\frac{s+\al}{2}}, \ \ t>0. 
\end{aligned}
\end{equation}

Finally, we would like to prove \eqref{lin-half1}. By virtue of the condition \eqref{cond-1}, we can see that for any $\e>0$, there exists $\delta>0$ such that if $|\xi|\le \delta$, then the following estimate holds, instead of \eqref{cond-2-re}: 
\begin{equation*}
\left|\wh{u}_{0}(\xi)-P(\xi)M_{0}\right| \les \e. 
\end{equation*}
Therefore, under the assumptions \eqref{cond-parameter-1-2}, in the same way to get \eqref{est-p>1} and \eqref{est-p=1}, for all $1\le p\le \infty$, we have the following estimate: 
\[
\left\||\N|^{s}\left(e^{\mu t\Delta}u_{0}-S_{\mu}(t)M_{0}\right)\right\|_{\wh{L}^{p}}
\les  \e t^{-\frac{d}{2p'}-\frac{s}{2}}+e^{-\frac{\mu}{2}\delta^{2}t}, \ \ t>0. 
\]
Moreover, taking the limit in the above, we obtain 
\[
\limsup_{t \to \infty}t^{\frac{d}{2p'}+\frac{s}{2}}\left\||\N|^{s}\left(e^{\mu t\Delta}u_{0}
-S_{\mu}(t)M_{0}\right)\right\|_{\wh{L}^{p}}
\les \e. 
\]
As a result, we eventually arrive at the desired asymptotic formula \eqref{lin-half1}, because $\e>0$ can be chosen arbitrary small. This completes the proof. 
\end{proof}
\section{Nonlinear Analysis}  

\indent

In this section, we shall carry out the nonlinear analysis of the target problem \eqref{eqn;NS}, to prove the main results. 
In what follows, for simplicity, we set the Duhamel term of \eqref{eqn;NS} as follows: 
\begin{equation}\label{DEF-N}
\mathcal{N}[u](x, t):=\int_0^t e^{\mu(t-\t)\Delta} \mathcal{P}_\s \div (u \otimes u)(\t) \,d\t. 
\end{equation}

\subsection{The 1st Order Asymptotics}  

\indent

This subsection is devoted to deriving the first asymptotic profile of the solution $u$ to \eqref{eqn;NS}. Namely, we shall give the proof of Theorems~\ref{thm;1st-1} and \ref{thm;1st-2}. By virtue of the linear analysis given in the previous section, in order to achieve this, it is sufficient to demonstrate the following decay estimate for the Duhamel term. The result \eqref{Duhamel-decay} below has already been obtained in \cite{N26} for the case of $d\ge3$. Here, let us carefully calculate it, including the two-dimensional case of $d=2$.
\begin{prop}[Decay of the Duhamel Term]\label{prop-N-decay}
Let $d\ge2$, $1\le p\le \infty$ and $s\in \R$. Suppose that the initial velocity $u_{0}$ satisfy $u_0 \in \widehat{\dot{B}}{_{\infty,2}^{-1}}(\R^d)\cap \wh{\dot{B}}{_{1,\infty}^0}(\R^d)$ with 
$\div u_0 =0$ in $\mathcal{S}'$, and $\|u_0\|_{\widehat{\dot{B}}{_{\infty,2}^{-1}}}$ is sufficiently small. 
Then, for $\mathcal{N}[u](x, t)$ defined by \eqref{DEF-N}, the following decay estimate holds: 
\begin{equation}\label{Duhamel-decay}
\left\||\N|^{s}\mathcal{N}[u](t)\right\|_{\wh{L}^{p}}
\les 
\begin{cases}
t^{-\frac{d}{2}(1-\frac{1}{p})-\frac{s+1}{2}}, &\text{if} \quad d\ge3, \\
t^{-\frac{d}{2}(1-\frac{1}{p})-\frac{s+1}{2}}\log (1+t), &\text{if} \quad d=2, 
\end{cases}
\end{equation}
for all $t\ge1$, where the parameters $d$, $p$ and $s$ fulfill the following condition: 
\begin{equation}\label{cond-s-NL}
s>-1-\min\left\{\frac{d}{p}, \frac{d}{p'}\right\} \quad \text{if} \quad 1<p<\infty, \qquad 
s\ge-1 \quad \text{if} \quad p=1, \infty. 
\end{equation}
\end{prop}
\begin{proof}
First of all, for $t\ge1$, let us split the integral in \eqref{DEF-N} as follows: 
\begin{align}\label{N-split}
\mathcal{N}[u](x, t)
=\left(\int_0^{t/2}+\int_{t/2}^t \right) e^{\mu(t-\t)\Delta} \mathcal{P}_\s \div (u \otimes u) \,d\t
=:\mathcal{N}_{1}[u](x, t)+\mathcal{N}_{2}[u](x, t). 
\end{align}

In what follows, we shall evaluate $\mathcal{N}_{1}[u](x, t)$ and $\mathcal{N}_{2}[u](x, t)$. We start with evaluating $\mathcal{N}_{1}[u](x, t)$. For all $t\ge1$, it follows from Hausdorff--Young's inequality, Proposition~\ref{thm;GWP} and \eqref{est;LpL1-re} that  
\begin{equation} \label{N1-est}
\begin{aligned}
\left\||\N|^{s}\mathcal{N}_{1}[u](t)\right\|_{\wh{L}^{p}}
\les&
\int_0^{t/2}\<t-\t\>^{-\frac{d}{2p'}-\frac{s+1}{2}}\|(u\otimes u)(\t)\|_{\wh{L}^1}d\t \\
\les& \<t\>^{-\frac{d}{2p'}-\frac{s+1}{2}}
     \int_0^1
       \left(\frac{\<t\>}{\<t-\t\>}\right)^{\frac{d}{2p'}+\frac{s+1}{2}}
       \|u(\t)\|_{\wh{L}^1}\|u(\t)\|_{\wh{L}^\infty}
      d\t \\
    &+t^{-\frac{d}{2p'}-\frac{s+1}{2}} \int_1^{t} \|u(\t)\|_{\wh{L}^1}\|u(\t)\|_{\wh{L}^\infty} d\t \\
\les& t^{-\frac{d}{2p'}-\frac{s+1}{2}} 
      \left(\|u\|_{L^\infty(0,1;\wh{L}^1)}\|u\|_{L^2(0,1;\wh{L}^\infty)}
      +\int_1^{t}\t^{-\frac{d}{2}} d\t \right) \\
\les& t^{-\frac{d}{2p'}-\frac{s+1}{2}} + t^{-\frac{d}{2p'}-\frac{s+1}{2}} 
\begin{cases}
1, &\text{if} \quad d\ge3, \\
\log t, &\text{if} \quad d=2
\end{cases} \\
\les& t^{-\frac{d}{2p'}-\frac{s+1}{2}} 
\begin{cases}
1, &\text{if} \quad d\ge3, \\
\log (1+t), &\text{if} \quad d=2. 
\end{cases} 
\end{aligned}
\end{equation}

For $\mathcal{N}_{2}[u](x, t)$ in \eqref{N-split}, it holds true that for all $d\ge2$ and $s$ satisfying \eqref{cond-s-NL}, 
\begin{align} \label{est;high}
\left\||\N|^{s}\mathcal{N}_{2}[u](t)\right\|_{\wh{L}^{p}} \les \int_{t/2}^t \|\div(u \otimes u)\|_{\widehat{\dot{H}}{_p^s}} d\t 
\les t^{-\frac{d}{2p'} -\frac{s+1}{2}-(\frac{d}{2}-1)}, \ \ t\ge1. 
\end{align}
In what follows, let us confirm \eqref{est;high}. 
First, we would like to deal with the case of $1<p<\infty$. 
Now, noticing that $\wh{\dot{H}}{_p^s}(\R^d) \simeq \wh{\dot{B}}{_{p,p'}^s}(\R^d)$ holds true, by virtue of Lemma~\ref{lem;equi}. 
Also, for all $s$ satisfying $-1-\min\{\frac{d}{p},\frac{d}{p'}\}<s<-1+\frac{d}{p}$, 
it follows from Lemma~\ref{lem;A-P} and Proposition~\ref{thm;decay} that 
\begin{align*}
\int_{t/2}^t \|\div(u \otimes u)\|_{\fB{^s_{p,p'}}} d\t 
&\les \int_{t/2}^t 
        \|u(\tau)\|_{\wh{L}^\infty \cap \wh{\dot{B}}{_{p,p'}^{\frac{d}{p}}}}
        \|u(\tau)\|_{\wh{\dot{B}}{_{p,p'}^{s+1}}} 
      d\t  \\
&\les \int_{t/2}^t \t^{-\frac{d}{2}} \t^{-\frac{d}{2p'}-\frac{s+1}{2}}d\t
\les t^{-\frac{d}{2p'}-\frac{s+1}{2}-(\frac{d}{2}-1)}. 
\end{align*}
On the other hand, if $s \ge -1+\frac{d}{p}$, namely $s$ satisfies $s+1 \ge \frac{d}{p}>0$, then we have from Lemma~\ref{lem;bil}, Proposition~\ref{thm;decay} and \eqref{est;LpL1-re} that 
\begin{align*}
\int_{t/2}^t \|\div(u \otimes u)\|_{\wh{\dot{B}}{_{p,p'}^s}} d\t 
&\les \int_{t/2}^t 
        \|u(\tau)\|_{\wh{L}^\infty} 
        \|u(\tau)\|_{\wh{\dot{B}}{_{p,p'}^{s+1}}} 
      d\t \\
&\les \int_{t/2}^t \t^{-\frac{d}{2}} \t^{-\frac{d}{2p'}-\frac{s+1}{2}}d\t
\les t^{-\frac{d}{2p'}-\frac{s+1}{2}-(\frac{d}{2}-1)}. 
\end{align*} 

Next, let us consider the case of $p=\infty$. If $s=-1$, then \eqref{est;LpL1-re} directly yields that 
\begin{align*}
\int_{t/2}^t \|\div(u \otimes u)\|_{\wh{\dot{H}}{_{\infty}^{-1}}} d\t  
&\les \int_{t/2}^t \|u(\tau)\|_{\wh{L}^\infty}^2 d\t  \\
&\les \int_{t/2}^t \t^{-\frac{d}{2}}\t^{-\frac{d}{2}}d\t 
\les t^{-\frac{d}{2}-(\frac{d}{2}-1)}. 
\end{align*}
As for the case of $s>-1$, thanks to the norm equivalence $\wh{\dot{H}}{_\infty^s}(\R^d) \simeq \wh{\dot{B}}{_{\infty,1}^s}(\R^d)$, using Lemma~\ref{lem;bil}, Proposition~\ref{thm;decay} and \eqref{est;LpL1-re} again, we can immediately see that 
\begin{align*}
\int_{t/2}^t \|\div (u \otimes u)\|_{\wh{\dot{B}}{_{\infty,1}^s}} d\t 
&\les \int_{t/2}^t \|u(\tau)\|_{\wh{L}^\infty} \|u(\tau)\|_{\wh{\dot{B}}{_{\infty,1}^{s+1}}} d\t \\
&\les  \int_{t/2}^t \t^{-\frac{d}{2}} \t^{-\frac{d}{2}-\frac{s+1}{2}}d\t 
\les t^{-\frac{d}{2}-\frac{s+1}{2}-(\frac{d}{2}-1)}. 
\end{align*}

Finally, we would like to derive \eqref{est;high} with $p=1$. For $s=-1$, using \eqref{est;LpL1-re} again, we analogously obtain that 
\begin{align*}
\int_{t/2}^t \|\div(u \otimes u)\|_{\wh{\dot{H}}{_{1}^{-1}}} d\t 
&\les \int_{t/2}^t \|u(\tau)\|_{\wh{L}^\infty}\|u(\tau)\|_{\wh{L}^1} d\t  \\
&\les \int_{t/2}^t \t^{-\frac{d}{2}}d\t 
\les t^{-(\frac{d}{2}-1)}. 
\end{align*}
Now, recalling that the fact $\wh{\dot{B}}{_{1, 1}^s}(\R^d) \hookrightarrow \wh{\dot{B}}{_{1, \infty}^s}(\R^d) \simeq \wh{\dot{H}}{_1^s}(\R^d)$ holds true. 
Hence, in the same way as before, for all $s>-1$, we have the following estimate: 
\begin{align*}
\int_{t/2}^t \|\div (u \otimes u)\|_{\wh{\dot{B}}{_{1, \infty}^s}} d\t 
&\les \int_{t/2}^t \|u(\tau)\|_{\wh{L}^\infty} \|u(\tau)\|_{\wh{\dot{B}}{_{1, \infty}^{s+1}}} d\t \\
&\les \int_{t/2}^t \t^{-\frac{d}{2}} \t^{-\frac{s+1}{2}}d\t 
\les t^{-\frac{s+1}{2}-(\frac{d}{2}-1)}. 
\end{align*}
Eventually, gathering all the above estimates, we can see that the estimate \eqref{est;high} is true. 
Therefore, combining \eqref{N-split}, \eqref{N1-est} and \eqref{est;high}, we can say that the desired result \eqref{Duhamel-decay} has been established. 
\end{proof}

\begin{proof}[\rm{\bf{End of the Proof of Theorems~\ref{thm;1st-1} and \ref{thm;1st-2}}}]
It follows from \eqref{IE} and \eqref{DEF-N} that 
\[
u(t)-S_{\mu}(t)M_{0}=\left(e^{\mu t\Delta}u_{0}-S_{\mu}(t)M_{0}\right)-\mathcal{N}[u](x, t). 
\]
First, summarizing up \eqref{cond-parameter-1-2} and \eqref{cond-s-NL}, then we obtain the condition ${\rm (C)}$. 
Therefore, combining the above equation, Propositions~\ref{prop-lin-half1-1} and \ref{prop-N-decay}, and taking the limit, we can conclude that the desired asymptotic formula \eqref{1st-1} is true. 
This completes the proof of Theorem~\ref{thm;1st-1}. 
In a similar way to the proof of Theorem \ref{thm;1st-1}, 
the desired decay estimate \eqref{1st-2} directly follows from the above equation, Propositions~\ref{prop-lin-half1-2} and \ref{prop-N-decay}. 
This completes the proof of Theorem~\ref{thm;1st-2}. 
\end{proof}

\subsection{The 2nd Order Asymptotics}  

\indent

Finally in this subsection, we would like to derive the second asymptotic profiles of the solution $u$ to \eqref{eqn;NS}. 
First, we shall treat the higher-dimensional case $d\ge3$. Namely, let us prove Theorem~\ref{thm;2nd-high}. The proof of the following proposition is similar to the proof of Theorem~2.4 in \cite{N26} by the second author. Here, we are able to give the proof of the following proposition, by slightly modifying the method used in \cite{N26}.
\begin{prop}[Asymptotic Profile for the Duhamel Term, Higher-dimensional Case]\label{prop-N-asymp-high}
Let $d\ge3$, $1\le p\le \infty$ and $s\in \R$. Suppose that the initial velocity $u_{0}$ satisfy $u_0 \in \widehat{\dot{B}}{_{\infty,2}^{-1}}(\R^d)\cap \wh{\dot{B}}{_{1,\infty}^0}(\R^d)$ with 
$\div u_0 =0$ in $\mathcal{S}'$, and $\|u_0\|_{\widehat{\dot{B}}{_{\infty,2}^{-1}}}$ is sufficiently small. 
Then, for $\mathcal{N}[u](x, t)$ defined by \eqref{DEF-N}, we have the following asymptotic formula: 
\begin{equation}\label{Duhamel-asymp-high}
\lim_{t\to \infty}t^{\frac{d}{2}(1-\frac{1}{p})+\frac{s+1}{2}}\left\||\N|^{s}\left(\mathcal{N}[u](t)-\p_{k}S_{\mu}(t)\int_{0}^{\infty}\int_{\R^{d}}(u_{k}u)(x, t)dxdt\right)\right\|_{\wh{L}^{p}}=0, 
\end{equation}
where the parameters $d$, $p$ and $s$ fulfill the condition \eqref{cond-s-NL}.  
\end{prop}
\begin{proof}
Throughout this proof, we shall use the same decomposition $\mathcal{N}_{1}[u](x, t)$ and $\mathcal{N}_{2}[u](x, t)$ for the Duhamel term $\mathcal{N}[u](x, t)$ as in \eqref{N-split}. Then, it immediately follows from the estimate \eqref{est;high} that the following result holds for all $d\ge3$: 
\begin{align} \label{est;high-2}
\limsup_{t\to \infty}t^{\frac{d}{2}(1-\frac{1}{p})+\frac{s+1}{2}}\left\||\N|^{s}\mathcal{N}_{2}[u](t)\right\|_{\wh{L}^{p}} 
\les \lim_{t\to \infty}t^{-(\frac{d}{2}-1)}=0. 
\end{align}

By virtue of the above estimate, we only need to analyze $\mathcal{N}_{1}[u](x, t)$. 
To do that, let us rewrite the target integral. Now, recalling the definition of $S_{\mu}(x, t)$ in \eqref{DEF-Stokes}, we have 
\begin{equation}
\begin{aligned}
&\mathcal{F}\left[\mathcal{N}_{1}[u](t)-\p_{k}S_{\mu}(t)\int_{0}^{\infty}\int_{\R^{d}}(u_{k}u)(x, t)dxdt\right](\xi)  \\
&=\int_0^{t/2} 
    i \xi_k e^{-\mu (t-\t)|\xi|^2} P(\xi) (2\pi)^{\frac{d}{2}}
    \mathcal{F}[(u_k u)](\xi, \t)d\t-i\xi_k \wh{S}_\mu(t) \int_0^\infty (2\pi)^{\frac{d}{2}}\mathcal{F}[(u_k u)](0, \t) d\t \\
&=\int_0^{t/2} 
    i \xi_k e^{-\mu (t-\t)|\xi|^2} P(\xi) (2\pi)^{\frac{d}{2}}
    \Big(\mathcal{F}[(u_k u)](\xi, \t)-\mathcal{F}[(u_k u)](0, \t)\Big) d\t \\
&\ \ \ \,-\int_0^{t/2} i\xi_k \left(e^{-\mu t|\xi|^2}-e^{-\mu(t-\t)|\xi|^2}\right) P(\xi) (2\pi)^{\frac{d}{2}} \mathcal{F}[(u_k u)](0, \t) d\t \\
&\ \ \ \,-i \xi_k \wh{S}_\mu(t) \int_{t/2}^\infty (2\pi)^{\frac{d}{2}}\mathcal{F}[(u_k u)](0, \t) d\t \\
&=: \wh{J}_{1}(\xi, t)+\wh{J}_{2}(\xi, t)+\wh{J}_{3}(\xi, t). \label{N1-split} 
\end{aligned}
\end{equation}
Thus, in order to obtain the desired formula \eqref{Duhamel-asymp-high}, it is sufficient to evaluate $J_{1}(x, t)$, $J_{2}(x, t)$ and $J_{3}(x, t)$. More precisely, the following decay estimates are required to complete the proof: 
\begin{align} \label{est;J123}
\lim_{t\to \infty}t^{\frac{d}{2}(1-\frac{1}{p})+\frac{s+1}{2}}\left\||\N|^{s}J_{k}(t)\right\|_{\wh{L}^{p}} =0 \quad (k=1, 2, 3). 
\end{align}

First, we would like to evaluate $J_{1}(x, t)$. Now, recalling that $\| |\eta|^{s+1} e^{-|\eta|^2} \|_{L^{p'}_\eta}
\les 1$ holds for all $1\le p\le \infty$ under the condition \eqref{cond-s-NL}. 
Then, using $|P(\xi)| \les 1$ again and performing the change of variables $\xi \mapsto \frac{\eta}{\sqrt{t-\t}}$, we can see that 
\begin{equation}
\begin{aligned}
\||\N|^sJ_{1}(t)\|_{\wh{L}^p}
&\les 
     \int_0^{t/2} 
     \left\||\xi|^{s+1}e^{-\mu(t-\t)|\xi|^2}\Big(\mathcal{F}[(u_k u)](\xi, \t)-\mathcal{F}[(u_k u)](0, \t)\Big)\right\|_{L^{p'}_\xi}
     d\t \\
     &\begin{aligned}
\les \int_0^{t/2} &(t-\t)^{-\frac{d}{2p'}-\frac{s+1}{2}} \\
&\times \left\||\eta|^{s+1}e^{-\mu|\eta|^2}\left(\mathcal{F}[(u_k u)]\left(\frac{\eta}{\sqrt{t-\t}}, \t\right)
       -\mathcal{F}[(u_k u)](0, \t)\right)
     \right\|_{L^{p'}_\eta}
     d\t \\
     \end{aligned} \\
&\les t^{-\frac{d}{2p'}-\frac{s+1}{2}} \int_0^{t/2} \psi_t(\t) d\t, \ \ t\ge1,   \label{J1-est}
\end{aligned}
\end{equation}
where $\psi_t(\t)
:=\|\eta|^{s+1}e^{-\mu|\eta|^2}(\mathcal{F}[(u_k u)](\frac{\eta}{\sqrt{t-\t}},\t)
       -\mathcal{F}[(u_k u)](0, \t))\|_{L^{p'}_\eta}$.  
To show the estimate \eqref{est;J123} with $k=1$, we need to confirm that 
\begin{equation} \label{eqn;so1}
\lim_{t \to \infty} \int_0^{t/2} \psi_t(\t)d\t=0.  
\end{equation}
In what follows, let us prove this. For any fixed $M>0$, can we easily see $\lim_{t \to \infty} \int_0^M \psi_t(\t)d\t=0$. 
Actually, it follows from Plancherel's theorem and the fact 
$u \in L^\infty(\R_+;\wh{L}^1) \cap L^2(\R_+;\wh{L}^\infty)$ (see, \eqref{est;LpL1-re} and Proposition~\ref{thm;GWP}) that 
\begin{align*}
 \psi_t(\t)
& \les \left\|\int_{\R^d} \left(1-e^{-iy\cdot \frac{\eta}{\sqrt{t-\t}}}\right) u_k(\t,y)u(\t,y)dy\right\|_{L^\infty_\eta} \\
&\les \|u(\t)\|_{L^2}^2 \les \|u(\t)\|_{\wh{L}^1}\|u(\t)\|_{\wh{L}^\infty} \in L^1_\t(0,M). 
\end{align*}
In addition, from the definition of $\psi_t(\t)$, we immediately have $\lim_{t \to \infty}\psi_t(\t) = 0$ for a.e.\,$\t \in (0,M)$. 
Therefore, we obtain $\lim_{t \to \infty} \int_0^M \psi_t(\t)d\t=0$ by virtue of the dominated convergence theorem. 
Moreover, for all $\ve>0$, we can choose $M>0$ such that $\int_M^\infty \psi_t (\t) d\t<\ve$ under the condition $d\ge3$, because 
it follows from Hausdorff--Young's inequality and \eqref{est;LpL1-re} that 
$$
\int_M^\infty \psi_t(\t) d\t
\les \int_M^\infty \|u(\t)\|_{L^2}^2 d\t
\les \int_M^\infty \t^{-\frac{d}{2}} d\t. 
$$
Gathering all the above results, we thus obtain that for all $\ve>0$, 
$$
\limsup_{t\to \infty}\int_0^{t/2} \psi_t(\t) d\t 
\les \lim_{t\to \infty}\int_0^M \psi_t(\t) d\t + \ve
=\ve. 
$$
Therefore, we could confirm that \eqref{eqn;so1} holds true, because $\ve>0$ can be chosen arbitrary small. 
As a result, combining \eqref{J1-est} and \eqref{eqn;so1}, we are able to obtain 
\[
\limsup_{t\to \infty}t^{\frac{d}{2p'}+\frac{s+1}{2}}\||\N|^sJ_{1}(t)\|_{\wh{L}^p}
\les \lim_{t \to \infty} \int_0^{t/2} \psi_t(\t)d\t=0. 
\]
This means that the estimate \eqref{est;J123} holds for $k=1$. 

Next, let us treat $J_{2}(x, t)$. 
Applying the mean value theorem to $e^{-\mu t|\xi|^2}-e^{-\mu(t-\t)|\xi|^2}$ in the target integral, we can see that 
\begin{equation} \label{J2-est-1}
\begin{aligned}
\||\N|^sJ_{2}(t)\|_{\wh{L}^p}
&\les 
     \int_0^{t/2} 
     \left\||\xi|^{s} i\xi_k \left(\int_0^1 \mu \t |\xi|^2e^{-\mu(t-\theta \t)|\xi|^2}d\theta\right)
     P(\xi) \mathcal{F}[(u_k u)](0, \t) \right\|_{L^{p'}_\xi}
     d\t  \\
&\les \int_0^{t/2} 
      \int_0^1 \left\||\xi|^{s+3} \t e^{-\mu(t-\theta \t)|\xi|^2}
     \mathcal{F}[(u_k u)](0, \t) \right\|_{L^{p'}_\xi}
     d\theta d\t  \\
&\les \int_0^{t/2} 
      \int_0^1 
      \t (t-\theta \t)^{-\frac{d}{2p'}-\frac{s+3}{2}}  
      \left\|\mathcal{F}[(u_k u)](0, \t) \right\|_{L^{\infty}_\xi}
     d\theta d\t \\
&\les t^{-\frac{d}{2p'}-\frac{s+3}{2}} \int_0^{t/2}\t \|u(\t)\|_{L^2}^2 d\t. 
\end{aligned}
\end{equation}
Since Proposition~\ref{thm;GWP} and the decay estimate \eqref{est;LpL1-re} with $p=2$, we obtain that for all $t\ge1$, 
\begin{equation}
\begin{aligned}
\int_0^{t/2} \t \|u(\t)\|_{L^2}^2\,d\t
&\les \int_0^1 \t \|u(\t)\|_{L^2}^2\,d\t + \int_1^t \t \|u(\t)\|_{L^2}^2\,d\t \\
&\les \|u\|_{L^\infty(0,1;\wh{L}^1)} \int_0^1 \t \|u(\t)\|_{\wh{L}^\infty} d\t 
     + \int_1^t \t \cdot \t^{-\frac{d}{2}} d\t \\
&\les \|u\|_{L^\infty(0,1;\wh{L}^1)} \|u\|_{L^2(0,1;\wh{L}^\infty)}+\int_1^t \t^{1-\frac{d}{2}} d\t 
\les 1+
\begin{cases}
t^{\frac{1}{2}}, &d=3, \\
\log t, &d=4, \\
t^{2-\frac{d}{2}}, &d\ge5.  
\end{cases}\label{J2-est-2}
\end{aligned}
\end{equation} 
Therefore, for all $t\ge1$, it follows from \eqref{J2-est-1} and \eqref{J2-est-2} that 
\[
\||\N|^sJ_{2}(t)\|_{\wh{L}^p} \les t^{-\frac{d}{2p'}-\frac{s+3}{2}}+t^{-\frac{d}{2p'}-\frac{s+1}{2}}
\begin{cases}
t^{-\frac{1}{2}}, &d=3, \\
t^{-1}\log t, &d=4, \\
t^{-(\frac{d}{2}-1)}, &d\ge5.  
\end{cases} 
\]
Taking the limit in the above, we can say that the estimate \eqref{est;J123} with $k=2$ has been established. 

Finally, we shall deal with $J_{3}(x, t)$. By using the decay estimate \eqref{est;LpL1-re} with $p=2$ again, we can analogously evaluate it as follows: 
\begin{align*}
\||\N|^sJ_{3}(t)\|_{\wh{L}^p} 
&\les \left\||\xi|^{s+1}P(\xi)e^{-\mu t|\xi|^2}\right\|_{L^{p'}_\xi}
     \left|\int_{t/2}^\infty \mathcal{F}[(u_k u)](\t,0) d\t \right| \\
&\les t^{-\frac{d}{2p'}-\frac{s+1}{2}}\int_{t/2}^\infty \|u(\t)\|_{L^2}^2 d\t 
\les t^{-\frac{d}{2p'}-\frac{s+1}{2}}\int_{t/2}^\infty \t^{-\frac{d}{2}} d\t  \\
&\les t^{-\frac{d}{2p'}-\frac{s+1}{2}-(\frac{d}{2}-1)}, \ \ t\ge1. 
\end{align*}
Therefore, we obtain the estimate \eqref{est;J123} with $k=3$ in the same way as before. Gathering all the above results, we can conclude that \eqref{est;J123} holds for all $k=1, 2, 3$. 

Now, applying \eqref{est;J123} to \eqref{N1-split}, we eventually arrive at 
\[
\lim_{t\to \infty}t^{\frac{d}{2}(1-\frac{1}{p})+\frac{s+1}{2}}\left\||\N|^{s}\left(\mathcal{N}_{1}[u](t)-\p_{k}S_{\mu}(t)\int_{0}^{\infty}\int_{\R^{d}}(u_{k}u)(x, t)dxdt\right)\right\|_{\wh{L}^{p}}=0. 
\] 
Combining this result, \eqref{est;high-2} and \eqref{N-split}, we thus complete the proof of the desired formula \eqref{Duhamel-asymp-high}. 
\end{proof}

\begin{proof}[\rm{\bf{End of the Proof of Theorem~\ref{thm;2nd-high}}}]
First, we note that \eqref{IE} and \eqref{DEF-N} lead to 
\begin{align*}
&u(t)-S_{\mu}(t)M_{0}+\p_{k}S_{\mu}(t)\int_{0}^{\infty}\int_{\R^{d}}(u_{k}u)(x, t)dxdt\\
&=\left(e^{\mu t\Delta}u_{0}-S_{\mu}(t)M_{0}\right)-\left(\mathcal{N}[u](x, t)-\p_{k}S_{\mu}(t)\int_{0}^{\infty}\int_{\R^{d}}(u_{k}u)(x, t)dxdt\right).
\end{align*}
Therefore, in addition to the assumptions in Theorem~\ref{thm;1st-2}, if we additionally assume $\al>1$, the desired asymptotic formula \eqref{2nd-1} can be immediately obtained, by virtue of the above equation, Propositions~\ref{prop-lin-half1-2} and \ref{prop-N-asymp-high}. This completes the proof of Theorem~\ref{thm;2nd-high}. 
\end{proof}

At the end of this paper, we would like to deal with the two-dimensional case of $d=2$, i.e., we shall give the proof of Theorem~\ref{thm;2nd-2}. 
The following proposition is one of the main contributions of this paper.
\begin{prop}[Asymptotic Profile for the Duhamel Term, Two-dimensional Case]\label{prop-N-asymp-2}
Let $d=2$. Suppose that the all conditions as in Theorem~\ref{thm;1st-2} are satisfied. 
If we additionally assume $\al\ge1$, for $\mathcal{N}[u](x, t)$ defined by \eqref{DEF-N}, we have the following asymptotic formula: 
\begin{equation}\label{Duhamel-asymp-two}
\left\||\N|^{s}\left(\mathcal{N}[u](t)-(\log t)\,\p_{k}S_{\mu}(t)\mathcal{A}_{k}[M_{0}]\right)\right\|_{\wh{L}^{p}}
\les t^{-\frac{d}{2}(1-\frac{1}{p})-\frac{s+1}{2}}, \ \ t\ge2, 
\end{equation}
where $\mathcal{A}_{k}[M_{0}]$ is defined by \eqref{DEF-A}.  
\end{prop}
\begin{proof}
First, for $t\ge2$, let us decompose the Duhamel term in the integral equation \eqref{IE} as follows: 
\begin{equation}
\begin{aligned}
&\mathcal{F}\left[\int_0^t e^{\mu(t-\t)\Delta} \mathcal{P}_\s \div (u \otimes u)(\t) \,d\t\right](\xi) \\
&=\left(\int_0^{1}+\int_1^{t/2}+\int_{t/2}^{t}\right) 
    i \xi_k e^{-\mu (t-\t)|\xi|^2} P(\xi) (2\pi)^{\frac{d}{2}}\mathcal{F}[(u_k u)](\xi, \t)d\t  \\ 
&=\int_0^{1}i \xi_k e^{-\mu (t-\t)|\xi|^2} P(\xi) (2\pi)^{\frac{d}{2}}\mathcal{F}[(u_k u)](\xi, \t)d\t \\
     &\ \ \ \,+\int_{t/2}^{t}i \xi_k e^{-\mu (t-\t)|\xi|^2} P(\xi) (2\pi)^{\frac{d}{2}}\mathcal{F}[(u_k u)](\xi, \t)d\t  \\  
&\ \ \ \,+\int_1^{t/2}i \xi_k e^{-\mu (t-\t)|\xi|^2} P(\xi) (2\pi)^{\frac{d}{2}}\mathcal{F}\left[\left(u_k u-\left(S_{\mu}M_{0}\right)_{k}S_{\mu}M_{0}\right)\right](\xi, \t)d\t     \\
&\ \ \ \,+\int_1^{t/2}i \xi_k e^{-\mu (t-\t)|\xi|^2} P(\xi) (2\pi)^{\frac{d}{2}}\mathcal{F}\left[\left(\left(S_{\mu}M_{0}\right)_{k}S_{\mu}M_{0}\right)\right](\xi, \t)d\t   \\
&=:\wh{R}_{1}(\xi, t)+\wh{R}_{2}(\xi, t)+\wh{R}_{3}(\xi, t)+\wh{L}(\xi, t). \label{DEF-RL} 
\end{aligned}
\end{equation}
In what follows, $\wh{R}_{i}(\xi, t)$ $(i=1, 2, 3)$ are considered as remainder terms. We can obtain the leading of the Duhamel term from $\wh{L}(\xi, t)$. 
To do that, we shall further divide it into the following form: 
\begin{equation}
\begin{aligned}
\wh{L}(\xi, t)=&\int_1^{t/2}i \xi_k e^{-\mu (t-\t)|\xi|^2} P(\xi) (2\pi)^{\frac{d}{2}}\mathcal{F}\left[\left(\left(S_{\mu}M_{0}\right)_{k}S_{\mu}M_{0}\right)\right](0, \t)d\t \\
&\begin{aligned}
+\int_1^{t/2}& i \xi_k e^{-\mu (t-\t)|\xi|^2} P(\xi) (2\pi)^{\frac{d}{2}} \\
&\times \left\{\mathcal{F}\left[\left(S_{\mu}M_{0}\right)_{k}S_{\mu}M_{0}\right](\xi, \t) -\mathcal{F}\left[\left(S_{\mu}M_{0}\right)_{k}S_{\mu}M_{0}\right](0, \t) \right\}d\t  
\end{aligned}
\\
=:& \,\wh{W}(\xi, t)+\wh{Z}(\xi, t).    \label{DEF-WZ}
\end{aligned}
\end{equation}
Next, for the latter sake, we would like to show the equality 
\begin{equation}\label{int-A}
\int_{\R^{2}}\left(\left(S_{\mu}M_{0}\right)_{k}S_{\mu}M_{0}\right)(x, \tau)dx
=\frac{\pi}{16\mu \tau}\left(2M_{0}^{k}M_{0}+|M_{0}|^{2}e_{k}\right)
=\tau^{-1}\mathcal{A}_{k}[M_{0}], 
\end{equation}
where $\mathcal{A}_{k}[M_{0}]$ is defined by \eqref{DEF-A}. Now, using the definition of $S_{\mu}(x, t)$ in \eqref{DEF-Stokes}, Parseval's identity and the polar coordinates $\xi = r \omega$ with $\omega=(\omega_{1}, \omega_{2}) \in S^1$, we have
\begin{equation}
\begin{aligned}
\int_{\mathbb{R}^2}
\left(\left(S_\mu M_{0}\right)_k \left(S_\mu M_{0}\right)_j\right)(x,\tau)dx
&=
\int_{\mathbb{R}^2}
e^{-2\mu \t |\xi|^2}
\left(P(\xi)M_{0}\right)_k \left(P(\xi)M_{0}\right)_j d\xi \\
&=\int_0^\infty e^{-2\mu \t r^2} r\,dr
\int_{S^1}\left(P(\omega)M_{0}\right)_k \left(P(\omega)M_{0}\right)_j d\sigma(\omega) \\
&=\frac{1}{4\mu \tau}\int_{S^1}\left(P(\omega)M_{0}\right)_k \left(P(\omega)M_{0}\right)_j d\sigma(\omega)=:\frac{1}{4\mu \tau} I_{kj}. \label{DEF-Ikj}
\end{aligned}
\end{equation}
Then, we need to calculate $I_{kj}$ for $k, j=1, 2$. Here, it follows from $P(\omega)=I-\omega \otimes \omega$ that 
\[
\left(P(\omega)M_{0}\right)_k=M_{0}^{k}-\omega_{k}(\omega \cdot M_{0}). 
\]
Therefore, a direct calculation yields that 
\begin{equation}
\begin{aligned}
I_{kj}&=\int_{S^1}\left\{M_{0}^{k}-\omega_{k}(\omega \cdot M_{0})\right\}\left\{M_{0}^{j}-\omega_{j}(\omega \cdot M_{0})\right\}d\sigma(\omega)  \\
&=\int_{S^1}\left\{ M_{0}^{k}M_{0}^{j}-M_{0}^{k}\omega_{j}(\omega \cdot M_{0})-M_{0}^{j}\omega_{k}(\omega \cdot M_{0}) +\omega_{k}\omega_{j}(\omega \cdot M_{0})^{2}\right\}d\sigma(\omega)  \\
&=2\pi M_{0}^{k}M_{0}^{j}
-M_{0}^{k}\sum_{l=1}^{2}M_{0}^{l}\int_{S^1}\omega_{j}\omega_{l}d\sigma(\omega)
-M_{0}^{j}\sum_{l=1}^{2}M_{0}^{l}\int_{S^1}\omega_{k}\omega_{l}d\sigma(\omega) \\
&\ \ \ \,+\sum_{l, m=1}^{2}M_{0}^{l}M_{0}^{m}\int_{S^{1}}\omega_{k}\omega_{j}\omega_{l}\omega_{m}d\sigma(\omega) \\
&
\begin{aligned}
=&\, 2\pi M_{0}^{k}M_{0}^{j}-M_{0}^{k}\sum_{l=1}^{2}M_{0}^{l}\pi \delta_{jl}-M_{0}^{j}\sum_{l=1}^{2}M_{0}^{l}\pi \delta_{kl} \\
&+\frac{\pi}{4}\sum_{l, m=1}^{2}M_{0}^{l}M_{0}^{m}\left(\delta_{kj}\delta_{lm}+\delta_{kl}\delta_{jm}+\delta_{km}\delta_{jl} \right)
\end{aligned}\\
&=2\pi M_{0}^{k}M_{0}^{j}-\pi M_{0}^{k}M_{0}^{j}-\pi M_{0}^{j}M_{0}^{k}+\frac{\pi}{4}\left(\delta_{kj}|M_{0}|^{2}+M_{0}^{k}M_{0}^{j}+M_{0}^{j}M_{0}^{k}\right) \\
&=\frac{\pi}{4}\left(2M_{0}^{k}M_{0}+\delta_{kj}|M_{0}|^{2}\right). \label{int-kakudo}
\end{aligned}
\end{equation}
Combining \eqref{DEF-Ikj} and \eqref{int-kakudo}, and rewriting it into the vector form, we obtain \eqref{int-A}. 
By virtue of \eqref{int-A}, the leading term of the Duhamel term can be derived from $\wh{W}(\xi, t)$ in \eqref{DEF-WZ} as follows: 
\begin{equation}
\begin{aligned}
\wh{W}(\xi, t)=&\int_1^{t/2}i \xi_k e^{-\mu (t-\t)|\xi|^2} P(\xi) \left(\int_{\R^{2}}\left(\left(S_{\mu}M_{0}\right)_{k}S_{\mu}M_{0}\right)(x, \tau)dx\right)d\t  \\
=&\int_1^{t/2}i \xi_k e^{-\mu (t-\t)|\xi|^2} P(\xi) \tau^{-1}\mathcal{A}_{k}[M_{0}]d\tau  \\
=&\left(\int_{1}^{t/2}\tau^{-1}d\tau\right) \wh{\p_{k}S_{\mu}}(\xi, t)\mathcal{A}_{k}[M_{0}] \\
&+\int_1^{t/2}i \xi_k \left\{e^{-\mu (t-\t)|\xi|^2}-e^{-\mu t|\xi|^2} \right\}P(\xi) \tau^{-1}\mathcal{A}_{k}[M_{0}] d\tau \\
=&\, (\log t)\,\wh{\p_{k}S_{\mu}}(\xi, t)\mathcal{A}_{k}[M_{0}] 
   + \left(\log \frac{1}{2}\right)\wh{\p_{k}S_{\mu}}(\xi, t)\mathcal{A}_{k} [M_{0}] \\
   &+\int_1^{t/2}i \xi_k \left\{e^{-\mu (t-\t)|\xi|^2}-e^{-\mu t|\xi|^2} \right\}P(\xi) \tau^{-1}\mathcal{A}_{k}[M_{0}] d\tau  \\
=:&\, (\log t)\,\wh{\p_{k}S_{\mu}}(\xi, t)\mathcal{A}_{k}[M_{0}]+\wh{W}_{1}(\xi, t)+\wh{W}_{2}(\xi, t). \label{W-split}
\end{aligned}
\end{equation}

Gathering the above results \eqref{DEF-N}, \eqref{DEF-RL}, \eqref{DEF-WZ} and \eqref{W-split}, we eventually arrive at 
\begin{equation}
\begin{aligned}\label{d=2-final}
&\mathcal{F}\Big[\mathcal{N}[u](t)-(\log t)\,\p_{k}S_{\mu}(t)\mathcal{A}_{k}[M_{0}]\Big](\xi) \\
&=\wh{R}_{1}(\xi, t)+\wh{R}_{2}(\xi, t)+\wh{R}_{3}(\xi, t)+\wh{W}_{1}(\xi, t)+\wh{W}_{2}(\xi, t)+\wh{Z}(\xi, t). 
\end{aligned}
\end{equation}
Therefore, in order to complete the proof, we need to evaluate the all remainder terms in the right-hand side of the above. 
However, noticing that the estimates for $\wh{R}_{1}(\xi, t)$ and $\wh{R}_{2}(\xi, t)$ have already been obtained in the previous discussion. 
Actually, we are able to evaluate $\wh{R}_{1}(\xi, t)$ in completely the same way to get \eqref{N1-est}. More precisely, we have 
\begin{align}
\left\||\N|^{s}R_{1}(t)\right\|_{\wh{L}^{p}}
\les& t^{-\frac{d}{2p'}-\frac{s+1}{2}}, \ \ t\ge2.  \label{R1-est}
\end{align}
Moreover, it immediately follows from \eqref{est;high} that 
\begin{align}\label{R2-est}
\left\||\N|^{s}R_{2}(t)\right\|_{\wh{L}^{p}} \les t^{-\frac{d}{2p'} -\frac{s+1}{2}-(\frac{d}{2}-1)}, \ \ t\ge2, 
\end{align}
because $R_{2}(x, t)\equiv \mathcal{N}_{2}[u](x, t)$ holds from their definitions \eqref{DEF-RL} and \eqref{N-split}. 
Thus, it is sufficient to evaluate $\wh{R}_{3}(\xi, t)$, $\wh{W}_{1}(\xi, t)$, $\wh{W}_{2}(\xi, t)$, $\wh{Z}(\xi, t)$ in what follows. 

Now, we shall derive the estimate for $\wh{R}_{3}(\xi, t)$. First, let us take $0<\delta<1/2$ arbitrary small. Then, we note that $\log(1+\tau) \les \t^{\delta}$ holds for all $\tau \ge1$. Therefore, if $\al \ge1$, it follows from Hausdorff--Young's inequality, Theorem~\ref{thm;1st-2}, \eqref{est;LpL1-re} and \eqref{DEF-Stokes} that 
\begin{align*}
&\left\|\left(u_k u-\left(S_{\mu}M_{0}\right)_{k}S_{\mu}M_{0}\right)(\tau)\right\|_{\wh{L}^{1}} \\
&\les \left\|\left\{\left(u-S_{\mu}M_{0}\right)_{k}u\right\}(\tau)\right\|_{\wh{L}^{1}} 
+\left\|\left\{\left(S_{\mu}M_{0}\right)_{k}  \left(u-S_{\mu}M_{0}\right)\right\}(\tau)\right\|_{\wh{L}^{1}} \\
&\les \left\|u(\tau)-S_{\mu}(\tau)M_{0} \right\|_{\wh{L}^{\infty}} \left\{ \left\|u(\tau)\right\|_{\wh{L}^{1}}+|M_{0}|\left\|S_{\mu}(\tau)\right\|_{\wh{L}^{1}}\right\} \\
&\les \t^{-1}\left\{\t^{-\frac{\al}{2}}+\t^{-\frac{1}{2}}\log (1+\t)\right\}
\les \t^{-\frac{3}{2}+\delta}, \ \ \t \ge1. 
\end{align*}
Using this estimate, analogously as \eqref{N1-est}, we are able to see that 
\begin{equation}
\begin{aligned}
\left\||\N|^{s}R_{3}(t)\right\|_{\wh{L}^{p}} 
&\les \int_1^{t/2}\<t-\t\>^{-\frac{d}{2p'}-\frac{s+1}{2}}\left\|\left(u_k u-\left(S_{\mu}M_{0}\right)_{k}S_{\mu}M_{0}\right)(\tau)\right\|_{\wh{L}^{1}} d\tau \\
&\les t^{-\frac{d}{2p'}-\frac{s+1}{2}}\int_{1}^{t/2}\t^{-\frac{3}{2}+\delta}d\tau 
\les t^{-\frac{d}{2p'}-\frac{s+1}{2}}, \ \ t\ge2. \label{R3-est}
\end{aligned}
\end{equation}

Next, we would like to deal with $\wh{W}_{1}(\xi, t)$ and $\wh{W}_{2}(\xi, t)$. For $\wh{W}_{1}(\xi, t)$, it is easy to see that 
\begin{align}\label{W1-est}
\left\||\N|^{s}W_{1}(t)\right\|_{\wh{L}^{p}} 
\les \left|\mathcal{A}_{k}[M_{0}] \right|\left\||\N|^{s}\p_{k}S_{\mu}(t)\right\|_{\wh{L}^{p}} 
\les \left|M_{0}\right|^{2} t^{-\frac{d}{2p'} -\frac{s+1}{2}}, \ \ t\ge2. 
\end{align}
On the other hand, in the similar way to get \eqref{J2-est-1}, we can evaluate $\wh{W}_{2}(\xi, t)$ as follows: 
\begin{equation} 
\begin{aligned}
\||\N|^sW_{2}(t)\|_{\wh{L}^p}
&\les 
     \int_0^{t/2} 
     \left\||\xi|^{s} i\xi_k \left(\int_0^1 \mu \t |\xi|^2e^{-\mu(t-\theta \t)|\xi|^2}d\theta\right)
     P(\xi) \tau^{-1}\mathcal{A}_{k}[M_{0}] \right\|_{L^{p'}_\xi}
     d\t  \\
&\les \left|\mathcal{A}_{k}[M_{0}] \right|\int_0^{t/2} 
      \int_0^1 \left\||\xi|^{s+3} e^{-\mu(t-\theta \t)|\xi|^2}\right\|_{L^{p'}_\xi}
     d\theta d\t  \\
&\les \left|M_{0}\right|^{2} \int_0^{t/2} 
      \int_0^1 
      (t-\theta \t)^{-\frac{d}{2p'}-\frac{s+3}{2}}  
     d\theta d\t 
     \les \left|M_{0}\right|^{2} t^{-\frac{d}{2p'} -\frac{s+1}{2}}, \ \ t\ge2. \label{W2-est}
\end{aligned}
\end{equation}

Finally, let us treat $\wh{Z}(\xi, t)$. In order to do that, we need to prepare an estimate below. Recalling the weighted estimate for the kernel $S_{\mu}(x, t)$, i.e., Lemma~\ref{lem;S}, we have 
\begin{align*}
&\left|(2\pi)^{\frac{d}{2}}\left\{\mathcal{F}\left[\left(S_{\mu}M_{0}\right)_{k}S_{\mu}M_{0}\right](\xi, \t) -\mathcal{F}\left[\left(S_{\mu}M_{0}\right)_{k}S_{\mu}M_{0}\right](0, \t) \right\}\right| \\
&=\left| \int_{\R^{2}}\left(e^{-ix\cdot \xi}-1\right)\left(\left(S_{\mu}M_{0}\right)_{k}S_{\mu}M_{0}\right)(x, \tau)dx\right| \\
&=\left| \int_{\R^{2}}\left(\int_{0}^{1}e^{-i\theta x\cdot \xi}d\theta \right)ix\cdot \xi \left(\left(S_{\mu}M_{0}\right)_{k}S_{\mu}M_{0}\right)(x, \tau)dx\right| \\
&\les \left|M_{0}\right|^{2} |\xi| \int_{\R^{2}}|x|\left|S_{\mu}(x, \tau)\right|^{2}dx 
\les \left|M_{0}\right|^{2}|\xi| \tau^{-\frac{1}{2}}, \ \ \xi \in \R^{d}, \ \t>0. 
\end{align*}
Therefore, a direct calculation gives us that 
\begin{equation}
\begin{aligned}
\left\||\N|^{s}Z(t)\right\|_{\wh{L}^{p}}
&\les \left|M_{0}\right|^{2}\int_{1}^{t/2}\left\||\xi|^{s+2}e^{-\mu(t-\tau)|\xi|^{2}}\right\|_{L^{p'}_{\xi}}\tau^{-\frac{1}{2}}d\tau  \\
&\les \left|M_{0}\right|^{2}\int_{1}^{t/2}(t-\tau)^{-\frac{d}{2p'}-\frac{s+2}{2}}\tau^{-\frac{1}{2}}d\tau 
\les  \left|M_{0}\right|^{2}t^{-\frac{d}{2p'}-\frac{s+1}{2}}, \ \ t\ge2. \label{Z-est}
\end{aligned}
\end{equation}

Eventually, summarizing the all results from \eqref{d=2-final} through \eqref{Z-est}, we are able to conclude that the desired asymptotic formula \eqref{Duhamel-asymp-two} holds true. This completes the proof. 
\end{proof}

\begin{proof}[\rm{\bf{End of the Proof of Theorem~\ref{thm;2nd-2}}}]
It follows from \eqref{IE} and \eqref{DEF-N} that 
\begin{align*}
&u(t)-S_{\mu}(t)M_{0}+(\log t)\,\p_{k}S_{\mu}(t)\mathcal{A}_{k}[M_{0}]\\
&=\left(e^{\mu t\Delta}u_{0}-S_{\mu}(t)M_{0}\right)-\left\{\mathcal{N}[u](x, t)-(\log t)\,\p_{k}S_{\mu}(t)\mathcal{A}_{k}[M_{0}]\right\}.
\end{align*}
As a result, under the assumptions in Theorem~\ref{thm;1st-2} and the additional condition $\al \ge1$, we are able to derive the desired asymptotic formula \eqref{2nd-2}, by using the above equation, Propositions~\ref{prop-lin-half1-2} and \ref{prop-N-asymp-2}. This completes the proof of Theorem~\ref{thm;2nd-2}. 
\end{proof}

\section*{Acknowledgments}

\indent

This study is supported by Grant-in-Aid for Young Scientists Research No.22K13936 and No.22K13939, Japan Society for the Promotion of Science.


\section*{Declaration of Generative AI and AI-assisted Technologies}

\indent

During the preparation of this work, the authors partially used ChatGPT (OpenAI) for language editing and for assistance with some routine calculations and computational checks. All mathematical arguments, calculations and results were independently verified by the authors, who take full responsibility for the content of the manuscript.




\medskip
\par\noindent
\begin{flushleft}
Ikki Fukuda\\
Division of Mathematics and Physics, \\
Faculty of Engineering, \\
Shinshu University, \\
4-17-1, Wakasato, Nagano, 380-8553, JAPAN\\
E-mail: i\_fukuda@shinshu-u.ac.jp

\bigskip
Ryosuke Nakasato\\
Division of Mathematics and Physics, \\
Faculty of Engineering, \\
Shinshu University, \\
4-17-1, Wakasato, Nagano, 380-8553, JAPAN\\
E-mail: nakasato@shinshu-u.ac.jp
\end{flushleft}

\end{document}